\documentclass[11pt,a4paper]{scrartcl}
\usepackage[english]{babel}
\usepackage{a4wide}
\usepackage[T1]{fontenc}
\usepackage[utf8]{inputenc}
\usepackage{amsmath, amsthm, amssymb}
\usepackage[numbers]{natbib}
\usepackage{graphicx} 
\usepackage{url}
\usepackage{hyperref}
\usepackage{xcolor}
\usepackage{booktabs}

\newcommand{\norm}[1]{\lVert #1 \rVert}
\newcommand{\sprod}[2]{\langle #1, #2 \rangle}
\newcommand{\Op}[0]{\mathcal{L}}
\newcommand{\R}[0]{\mathbb{R}}
\newcommand{\N}[0]{\mathbb{N}}
\newcommand{\C}[0]{\mathbb{C}}

\newcommand{\ftau}[0]{f^{\boldsymbol{\tau}}}
\newcommand{\ml}[0]{\mathrm{e}_{\alpha,\beta}}

\newtheorem{theorem}{Theorem}
\newtheorem{lemma}{Lemma}
\newtheorem{remark}{Remark}

\newtheorem{definition}{Definition}
\newtheorem{algorithm}{Algorithm}

\author{Tobias Danczul\thanks{Institute for Analysis and Scientific Computing (TU Wien),
		Wiedner Hauptstrasse 8-10, 1040 Wien, Austria.}
	\and
	Clemens Hofreither\thanks{Johann Radon Insitute for Computational and Applied Mathematics (RICAM), Altenbergerstr.~69, 4040 Linz, Austria.}
	\and
	Joachim Sch\"oberl\footnotemark[1]
}

\title{A Unified Rational Krylov Method for Elliptic and Parabolic Fractional Diffusion Problems}
\begin{document}
	
\maketitle

\begin{abstract}
	We present a unified framework to efficiently approximate solutions to fractional diffusion problems of stationary and parabolic type. After discretization, we can take the point of view that the solution is obtained by a matrix-vector product of the form $\ftau(L)\mathbf{b}$, where $L$ is the discretization matrix of the spatial operator, $\mathbf{b}$ a prescribed vector, and $\ftau$ a parametric function, such as a fractional power or the Mittag-Leffler function. In the abstract framework of Stieltjes and complete Bernstein functions, to which the functions we are interested in belong to, we apply a rational Krylov method and prove uniform convergence when using poles based on Zolotar\"ev's minimal deviation problem. The latter are particularly suited for fractional diffusion as they allow for an efficient query of the map $\boldsymbol{\tau}\mapsto \ftau(L)\mathbf{b}$ and do not degenerate as the fractional parameters approach zero.	
	
	We also present a variety of both novel and existing pole selection strategies for which we develop a computable error certificate. Our numerical experiments comprise a detailed parameter study of space-time fractional diffusion problems and compare the performance of the poles with the ones predicted by our certificate.
	
\end{abstract}
\textit{Keywords}: fractional diffusion, numerical methods, rational Krylov method, rational approximation, Zolotar\"ev problem

\section{Introduction}
The theoretical groundings of fractional PDEs together with their confirmation in scientific experiments has sparked a remarkable amount of research activity across all engineering disciplines. Their field of application is broadly scattered \cite{Sun2018} and comprises material science, image processing, and finance, to name only a few.

The precise derivation of fractional powers of differential operators is a nontrivial task in itself. Several competing definitions are advocated by the literature \cite{Lischke2020} among which we focus on the spectral definition in space and the Caputo fractional derivative in time.

Provided any popular discretization scheme, such as the finite element or finite difference method, the approximation of the model under investigation typically boils down to the computation of a matrix-vector product of the form $\ftau(L)\mathbf{b}$, where
\begin{itemize}
    \item $L\in\R^{N\times N}$ is a sparse, positive definite discretization matrix of the integer-order differential operator in question,
	\item $\textbf{b}\in\R^N$ the coefficient vector of some prescribed data,
	\item $\ftau$ a matrix function that corresponds to the particular problem at hand and depends on a collection of parameters encoded in the vector $\boldsymbol{\tau}\in\Theta\subset\R^d$, $d\in\mathbb{N}$. Throughout this manuscript, we focus on the following configurations of $\ftau$.
		\begin{itemize}
            \item The scenario $\ftau(\lambda) = f^s(\lambda) = \lambda^s$, $s\in\Theta=[0,1]$, corresponds to the application of a fractional operator and is of interest in, e.g., time stepping schemes for time-dependent problems generated by a nonlocal operator in space.
			\item In stationary fractional diffusion problems one is interested in the choice $\ftau(\lambda) = f^s(\lambda) = \lambda^{-s}$, $s\in\Theta=[0,1]$.
			\item Discrete solutions to homogeneous parabolic problems involving a fractional operator in space at a fixed time $t>0$ can be concisely written via $e^{-tL^s}\mathbf{b}$, such that $\ftau(\lambda) = e^{-t\lambda^s}$, $\boldsymbol{\tau} = (t,s)\in\Theta=\R^+\times[0,1]$.
			\item A generalization of the latter is provided by space-time fractional diffusion problems, where the exponential is replaced by the generalized Mittag-Leffler function
			\begin{align*}
				E_{\alpha,\beta}(\lambda) = \sum_{k=0}^\infty \frac{\lambda^k}{\Gamma(\alpha k + \beta)},
			\end{align*}
			$\Gamma$ referring to the gamma function. In this scenario one has $\ftau(\lambda) = E_{\alpha,\beta}(-t^{\alpha}\lambda^s)$ with $\boldsymbol{\tau} = (\alpha, \beta, t, s)\in\Theta = (0,1]\times[\alpha,\infty)\times\R^+\times[0,1]$.
		\end{itemize}   
\end{itemize}
The \textit{exact} evaluation of $\ftau(L)\mathbf{b}$ typically involves the computation of all eigenvectors of $L$ and quickly exceeds a reasonable amount of computational capacity. Direct rational approximation methods have been employed in e.g., \cite{Harizanov2018, BURA:2019,Pasciak2020,Vabishchevich2020} to alleviate the computational costs; see also \cite{Wohlmuth2021}. The idea is to replace $\ftau$ by a suitable rational function $r^{\boldsymbol{\tau}}$ such that $r^{\boldsymbol{\tau}}(L)\mathbf{b}\approx\ftau(L)\mathbf{b}$. Although elementary, this concept can by found in a variety of modern approximation schemes for the stationary fractional diffusion problem \cite{H:2020,DH2021} among which we discuss a few in the following.

One class of methods is based on the prominent Caffarelli-Silvestre extension \cite{Caffarelli2007,Stinga2010,Cabre2010,Capella2011,Braendle2013}, whose susceptibility to standard finite element methods has evoked a large amount of research activity \cite{Nochetto2015,Ainsworth2018,Banjai2018}. Originally proposed for the fractional Laplacian, these schemes allow for generalizations to the time-dependent regime \cite{Nochetto2016, Melenk2020}. 

Initiated by \cite{Bonito2015}, a number of quadrature methods have been presented for the stationary fractional diffusion problem \cite{Bonito2018b,Bonito2020,Antil2019}; see also \cite{Li2017,Li2018}. The idea is to rewrite $\ftau(L)$ via Cauchy's formula as a contour integral over a parametrized family of local problems. If $\ftau(\lambda) = \lambda^{-s}$, the contour can be chosen as the negative real line, in which case the integral representation is known as Balakrishnan's formula \cite{Balakrishnan1960}. This approach has also been adapted to parabolic problems \cite{Bonito2017,Bonito2017b,Rieder2020}. Since real-valued integration paths cause highly oscillatory behaviour, however, one typically resorts to complex contours which in turn necessitates solutions to complex-valued problems, even if $L$ and $\mathbf b$ are real.

Another approach for approximating $\ftau(L)\mathbf{b}$ are the rational Krylov methods (RKM) \cite{Novati2017, Aceto2019, Novati2011, Benzi2020}, to which, in view of \cite{DH2021}, a variety of reduced basis methods \cite{Antil2019,Bonito2020,DS:2019,DS:2020} belong. The RKM extracts a low-dimensional surrogate $\mathbf{u}_{k+1}\approx \ftau(L)\mathbf{b}$ from a search space of the form $\mathcal{Q}^\Xi_{k+1}(L,\mathbf{b}) := \operatorname{span}\{q_k(L)^{-1}\mathbf{b},\dots,L^kq_k(L)^{-1}\mathbf{b}\}$, where $\Xi = \{\xi_1,\dots,\xi_k\}\subset\C$ is a set of a-priori selected parameters, the so-called \textit{poles}, and $q_k$ the monic polynomial of degree $k$ with roots in $\Xi$. The question of optimal poles highly depends on the particular function $\ftau$ and thus also on the involved parameters encoded in $\boldsymbol{\tau}$. In light of the fact that fractional parameters are typically employed to fit the mathematical model to the observed data \cite{Sprekels2016, Grau2014}, it is desirable to choose the poles independently of these quantities. Furthermore, if both $L$ and $\mathbf{b}$ are real, it is worthwhile to restrict $\Xi$ to the real line to avoid complex arithmetic.

In this paper, we present and analyze both novel and existing pole selection strategies for RKMs to approximate the family of parametric matrix-vector products $\{\ftau(L)\mathbf{b}:\tau\in\Theta\}$ uniformly in $\boldsymbol{\tau}$. To this end, we show that the functions we are interested in are either of Laplace-Stieltjes, Cauchy-Stieltjes, or complete Bernstein type, i.e., they admit a representation of the form
\begin{align}
\label{eq:stieltjes}
	\ftau(\lambda) = \int_{0}^{\infty} \mu^{\boldsymbol{\tau}}(\zeta)g(\zeta,\lambda)\, d\zeta,\qquad g(\zeta,\lambda) \in \{e^{-\zeta\lambda}, (\zeta+\lambda)^{-1}, \lambda(\zeta+\lambda)^{-1}\},
\end{align}
where $\mu^{\boldsymbol{\tau}}$ is a real valued function such that the integral is absolutely convergent. Inspired by \cite{Robol2020}, we provide, in this abstract framework, a certified upper bound for the rational Krylov error in dependence of its poles $\Xi = \{\xi_1,\dots, \xi_k\}$. Its rate of decay for increasing values of $k$ is dictated by the maximal deviation of
\begin{align*}
	r_\Xi(\lambda) = \prod_{j=1}^k \frac{\lambda + \xi_j}{\lambda - \xi_j}
\end{align*}
over the spectral interval $\Sigma$ of $L$. Minimizing $\norm{r_\Xi}_{L_\infty(\Sigma)}$ leads to Zolotar\"ev's well-known minimal deviation problem whose analytical solution provides a $\boldsymbol{\tau}$-independent selection of \textit{real} poles and allows for an efficient querying of the solution map $\boldsymbol{\tau} \mapsto \ftau(L)\mathbf{b}$. Extending the works \cite{DS:2019,DS:2020,DH2021}, where pointwise convergence in the parameter $s\in(-1,1)$ for the special case $\ftau(\lambda) = \lambda^s$ was proven, we show exponential convergence rates which are uniform in $\boldsymbol{\tau}$. 

A computational inconvenience of Zolotar\"ev's poles is the fact that they are not nested. Even though a variety of nested pole sequences exist, guaranteed error bounds are typically not available. To address this difficulty, we develop an algorithm to compute $\norm{r_\Xi}_{L_\infty(\Sigma)}$ which in turn allows us to assess the quality of these poles when no theoretical bounds for the error are available. Similarly to \cite{Druskin2010, Druskin2011, Guettel2013b}, we also present two novel pole selection algorithms by greedily minimizing $\norm{r_\Xi}_{L_\infty(\Sigma)}$. The first one only requires to inject the extremal eigenvalues of $L$ and automatically detects the global extrema of $r_\Xi$ in $\Sigma$. In each step, the error estimate is computed as a by-product and thus directly available. The second scheme is fully automatic and generates information about the spectral region without any user-provided data. In our numerical experiments, we perform a detailed parameter study to illuminate the impact of changing values of $\boldsymbol{\tau}$ on the Krylov approximation. We compare different pole selection strategies and discuss their performance with the one predicted by our error certificate. 

In view of \cite{DH2021}, the presented results can be seen as an improvement and extension of \cite{DS:2019, DS:2020}. Our analytical findings show that the results of \cite{Robol2020} admit a natural generalization to complete Bernstein functions. For a certain class of Laplace-Stieltjes functions, we further justify the experimental observation in \cite{Robol2020} that the RKM error decays with purely exponential convergence rates when Zolotar\"ev's pole distribution is used (cf. Theorem \ref{thm:lsimprovement}).

We end this introduction with an overview of the structure of this manuscript. In Section \ref{sec:fracdiff}, we establish a close relation between solutions to fractional diffusion problems and parametric families of matrix-vector products. We introduce the notion of Stieltjes and Bernstein theory and show that the functions of interest can be cast in this unified theoretical framework. After a concise survey of the RKM, we provide, in Section \ref{sec:RKM}, the essential theoretical preparations for the development of our numerical and analytical findings. The core of this paper is provided in Section \ref{sec:rkmfracdiff}, where the analytical key results, the error certificate, and the novel pole selection algorithms are presented. Finally, in Section \ref{sec:numerics}, we demonstrate the effectiveness of the developed tools by means of numerical experiments for a few space-time fractional diffusion problems. 
Some proofs are given in Appendix~\ref{sec:appendix}.

\section{The notion of Stieltjes and complete Bernstein functions in fractional diffusion}
\label{sec:fracdiff}
Let $\Omega\subset \R^d$, $d\in\mathbb{N}$, be a bounded Lipschitz domain, $D\in L_\infty(\Omega;\R^{d\times d})$ symmetric and uniformly positive definite, and $c\in L_\infty(\Omega)$ with $c\geq 0$ almost everywhere. We shall be concerned with the fractional powers of the self-adjoint operator
\begin{align}
	\label{eq:operator}
	\Op u := -\operatorname{div}(D\nabla u) + cu
\end{align}
defined by spectral expansion
\begin{align*}
	\Op^s u = \sum_{j=1}^\infty \lambda_j^s\sprod{u}{\varphi_j}_{L_2}\varphi_j,
\end{align*}
where $\sprod{\cdot}{\cdot}_{L_2}$ denotes the $L_2$-inner product on $\Omega$ and $(\lambda_j,\varphi_j)_{j=1}^\infty$ the collection of eigenvalues and eigenfunctions of \eqref{eq:operator} equipped with homogeneous Dirichlet boundary conditions. We are particularly interested in numerical approximations of solutions to fractional diffusion equations of the following form: Given $(\alpha,s)\in[0,1]^2$, $T>0$, a forcing term $b\in C([0,T], L_2(\Omega))$, and an initial datum $u_0\in L_2(\Omega)$, we seek $u:[0,T]\to H_0^s(\Omega)$ such that
\begin{align}
	\label{eq:FracSpaceTimeEq}
	\partial_t^{\alpha} u + \mathcal{L}^su = b \text{ in }\Omega\times (0,T], \quad u(t)|_{\partial\Omega} = 0\text{ in }(0,T], \quad u(0) = u_0 \text{ in }\Omega.
\end{align}
The fractional derivative in time $\partial_t^\alpha$ is understood as \textit{Caputo fractional derivative} of order $\alpha$ with respect to $t$ \cite{Podlubny1999}, that is,
\begin{align*}
	\partial_t^\alpha u(x,t) = \frac{1}{\Gamma(1-\alpha)}\int_0^t\frac{1}{(t-\tau)^\alpha}\frac{\partial u(x,\tau)}{\partial \tau}\, d\tau, \rlap{\quad\qquad$\alpha\in[0,1).$}
\end{align*}
By convention, we set $\partial_t^1 := \partial_t$. If $\alpha = 0$, $\partial_t^\alpha u = \partial_t^0 u = u$ and \eqref{eq:FracSpaceTimeEq} is understood in the stationary sense, in which case the initial condition is neglected.

Provided a finite element space $V_h\subset H_0^1(\Omega)$, we can approximate \eqref{eq:FracSpaceTimeEq} by a fractional differential equation obtained by the spatial discretization. For concreteness, let $(\psi_j^h)_{j=1}^N$ label a basis of $V_h$ and introduce the mass and stiffness matrix by
\begin{align}
	\label{eq:femmatrices}
	M_{ij} = \sprod{\psi_j^h}{\psi_i^h}_{L_2},\qquad A_{ij} = \sprod{D\nabla\psi_j^h}{\nabla\psi_i^h}_{L_2} + c\sprod{\psi_j^h}{\psi_i^h}_{L_2},
\end{align}
respectively. We utilize the discrete eigenfunction method to discretize the spatial fractional operator, i.e., we replace $\Op^{s}$ by its finite element approximation $L^{s}$, where $L := M^{-1}A$. The arising semi-discrete system of equations reads as
\begin{align}
	\label{eq:ODE}
	\partial_t^\alpha \mathbf{u}(t) + L^s\mathbf{u}(t) = \mathbf{b}(t), \quad \mathbf{u}(0) = \mathbf{u}_0,
\end{align} 
where $\mathbf{b}(t)$, $\mathbf{u}_0\in\R^N$ label the coefficient vector of the $L_2$-orthogonal projection of $b(t)$ and $u_0$ onto $V_h$, respectively. Solutions to \eqref{eq:ODE} can be concisely written as a matrix-vector product. Clearly, if $\alpha = 0$, we have $\mathbf{u} = (I + L^{s})^{-1}\mathbf{b}$ as claimed. If $\alpha > 0$, we can consult \cite[p. 140]{Podlubny1999} to affirm
\begin{align*}
	\mathbf{u}(t) = E_{\alpha,1}\left(-t^\alpha L^s\right)\mathbf{u}_0 + \int_0^t(t-\zeta)^{\alpha -1}E_{\alpha,\alpha}\left(-(t-\zeta)^\alpha L^s\right)\mathbf{b}(\zeta)\, d\zeta.
\end{align*}
If $\mathbf{b}(t)$ is a vector-valued polynomial in $t$, the integral can be reinterpreted as Mittag-Leffler function itself. As indicated by \cite{Novati2011}, there holds for any $\mathbf{b}(t) = \sum_{k=0}^p t^k\mathbf{v}_k$, $\mathbf{v}_k\in\R^N$,
\begin{align*}
	\mathbf{u}(t) = E_{\alpha,1}\left(-t^\alpha L^s\right)\mathbf{u}_0 + \sum_{k=0}^p\Gamma(k+1)t^{\alpha+k}E_{\alpha,\alpha+k+1}\left(-t^\alpha L^s\right)\mathbf{v}_k.
\end{align*}
In particular, solutions to standard (local) parabolic problems are recovered if one sets $\alpha = s = 1$. In view of these results, our ambition is to propose an efficient approximation for the family of matrix-vector products $\{\ftau(L)\mathbf{b} : \mathbf{\tau}\in \Theta\}$, where $\ftau$ represents a prototypical function arising from fractional diffusion problems and $\boldsymbol{\tau}\in\Theta$ the collection of problem-specific parameters. An important theoretical tool we use throughout this manuscript is the fact that, depending on the particular problem, $\ftau$ admits an integral representation with a kernel of exponential or resolvent type.

\subsection{Stieltjes and complete Bernstein functions}
In this section, we investigate functions of type \eqref{eq:stieltjes}, see \cite{Schilling2009}, and establish a connection to fractional diffusion problems.
\begin{definition}
	A function $f:\R^+\to\R$ is said to be a Laplace-Stieltjes function if
	\begin{align}
		\label{eq:LaplaceStieljesInt}
		f(\lambda) = \int_0^\infty e^{-\zeta\lambda}\mu_L(\zeta)\, d\zeta
	\end{align}
	for some positive real-valued function $\mu_L$ such that the integral is absolutely convergent.
	We denote the set of all Laplace-Stieltjes function with $\mathcal{LS}$.
\end{definition}

\begin{definition}
	A function $f:\R^+\to\R$ is said to be a Cauchy-Stieltjes function if
	\begin{align}
		\label{eq:CauchyStieljesInt}
		f(\lambda) = \int_0^\infty \frac{\mu_C(\zeta)}{\zeta+\lambda}\, d\zeta
	\end{align}
	for some positive real-valued function $\mu_C$ such that the integral is absolutely convergent. We denote the set of all Cauchy-Stieltjes functions with $\mathcal{CS}$.
\end{definition}

An important observation is the fact that each Cauchy-Stieltjes function $f$ of the form \eqref{eq:CauchyStieljesInt} can be rewritten as a Laplace-Stieltjes function \eqref{eq:LaplaceStieljesInt} with 
\begin{align*}
	\mu_L(\zeta) = \int_0^\infty e^{-\zeta s}\mu_C(s)\, ds,
\end{align*}
see e.g., \cite{Schneider1996}. Therefore, we have $\mathcal{CS}\subset\mathcal{LS}$.

\begin{definition}
	A function $f:\R^+\to\R$ is said to be a complete Bernstein function if
	\begin{align*}
		f(\lambda) = \int_0^\infty \frac{\lambda}{\zeta + \lambda} \mu_B(\zeta) \,d\zeta
	\end{align*}
	for some positive real-valued function $\mu_B$ such that the integral is absolutely convergent. We denote the set of all complete Bernstein function with $\mathcal{CB}$ and define
	\begin{align*}
		\mathcal{SB} := \mathcal{LS} \,\cup\,\mathcal{CS}\,\cup\,\mathcal{CB} = \mathcal{LS} \,\cup\,\mathcal{CB}.
	\end{align*}
\end{definition}
For some $f\in\mathcal{SB}$ the corresponding density function $\mu(\zeta)$ can be computed explicitly. In order to apply the results provided in this paper, however, it suffices to know whether the desired function admits such a representation without the explicit knowledge of the integrand.
\begin{definition}
	An infinitely differentiable function $f:\R^+\to\R$ is said to be completely monotonic if
	\begin{align}
		\label{eq:monoton}
		(-1)^nf^{(n)}(\lambda)\geq 0, \rlap{\qquad$n\in\N_0,$}
	\end{align}
	for all $\lambda\in\R^+$. We denote the set of all completely monotonic functions with $\mathcal{CM}$.
\end{definition}
The theory of Stieltjes and complete Bernstein functions is an active field of research. Several generic strategies exist to determine whether $f$ is of the desired form, among which we present a few in the following lemma.
\begin{lemma}
	\label{lemma:properties}
	There holds
	\begin{itemize}
		\item[1.] $\mathcal{LS} = \mathcal{CM}$,
		\item[2.] $f(\lambda)\in\mathcal{CB}$ if and only if $f(\lambda) / \lambda \in\mathcal{CS}$,
		\item[3.] if $f\in\mathcal{CM}$ and $g\in\mathcal{CB}$, then $f\circ g\in\mathcal{CM}$,
		\item[4.] if $f\equiv c>0$, then $f\in\mathcal{LS}$.
	\end{itemize}
\end{lemma}
\begin{proof}
	The first claim is a well-known result of Bernstein \cite{Bernstein1929}, the second follows by definition, the third can be found in \cite{Schilling2009}. If $f\equiv c>0$, then $f$ satisfies \eqref{eq:monoton}, i.e., $f\in\mathcal{CM} = \mathcal{LS}$.
\end{proof}

In light of these results, we are now in position to affirm that the functions we are interested in satisfy integral representations of the aforementioned forms. For convenience, we define
\begin{align}
\label{eq:mlextended}
	\ml(t,\lambda) := 
	\begin{cases}
		E_{\alpha,\beta}(t\lambda), \quad&\text{if }\alpha > 0, \\
		(1+\lambda)^{-1}, \quad&\text{if }\alpha = 0.
	\end{cases}
\end{align} 
\begin{lemma}
	\label{lemma:fracstieltjes}
	$\,$
	\begin{itemize}
		\item[1.] For all $s\in(0,1)$ there holds $f(\lambda) = \lambda^{s}\in\mathcal{CB}$.
		\item[2.] For all $s\in[0,1]$ there holds $f(\lambda) = \lambda^{-s}\in\mathcal{LS}$. If $s\in(0,1)$, then $f\in\mathcal{CS}$.
		\item[3.] For all $\alpha\in[0,1]$, $\beta\geq\alpha$, $t\in\R_0^+$, and $s\in[0,1]$ there holds $f(\lambda) = \ml(-t^\alpha, \lambda^s)\in\mathcal{LS}$. If $t>0$ and $s>0$ with $s+\frac{\alpha}{2}<1$, then $f\in\mathcal{CS}$.
	\end{itemize}
\end{lemma}
\begin{proof}
	The first claim is a classical result and can be found in \cite{Schilling2009}. The second claim is a direct consequence of Balakrishnan's formula \cite{Balakrishnan1960}, $\lambda^0 = 1 >0$, $\lambda^{-1}\in\mathcal{CM}$, and Lemma \ref{lemma:properties}.
	
	To show the third claim, we start with $\alpha = 0$ and note that $(1+\lambda^0)^{-1} = 1/2\in\mathcal{CM}$ and $(1+\lambda)^{-1} \in\mathcal{CM}$. Again, due to Lemma \ref{lemma:properties}, $\mathrm{e}_{0,\beta}(-t^0,\lambda^s) \in\mathcal{LS}$ if $s\in\{0,1\}$. Moreover, arguing as in \cite[Proposition 4.3]{Turner2011}, one verifies $\mathrm{e}_{0,\beta}(-t^0,\lambda^s)\in\mathcal{CS}\subset\mathcal{LS}$ if $s\in(0,1)$. 
	
	Let now $\alpha > 0$. If $t = 0$, we have $f \equiv 1 \in\mathcal{CM}$ and the claim is valid for all $(\alpha,\beta,s)\in(0,1]\times[\alpha,\infty)\times[0,1]$. Let therefore $t>0$; w.l.o.g. we assume $t = 1$. As shown in \cite{Schneider1996, Miller1999}, $E_{\alpha,\beta}(-\lambda)\in\mathcal{CM}$ if and only if $\alpha\in(0,1]$ and $\beta\geq\alpha$. Since $E_{\alpha,\beta}(-\lambda^0)\equiv 1\in\mathcal{CM}$, the claim holds if $s\in\{0,1\}$. Due to $\lambda^s\in\mathcal{CB}$ for all $s\in(0,1)$, we deduce from the third property in Lemma \ref{lemma:properties} that $E_{\alpha,\beta}(-\lambda^s)\in\mathcal{CM}$. Thus, Lemma \ref{lemma:properties} can be consulted to affirm that $\ml(-1^\alpha,\lambda^s) = E_{\alpha,\beta}(-\lambda^s)\in\mathcal{LS}$.
	
	As shown in \cite[Proposition 4.5]{Turner2011} and \cite[Example 4.3]{Novati2017}, $E_{\alpha,\beta}(-\lambda^s)\in\mathcal{CS}$ if $s+\frac{\alpha}{2}<1$ and $s\in(\frac{1}{2},1]$. It is easily verified, however, that the proof remains valid if $s\in(0,\frac{1}{2}]$. 
\end{proof}

Lemma \ref{lemma:fracstieltjes} justifies to consider fractional diffusion problems in the abstract framework of Stieltjes and complete Bernstein functions. The results we provide in the sequel, however, apply to any function $f\in\mathcal{SB}$, such as
\begin{align*}
	\frac{\ln(1+\lambda)}{\lambda} \in\mathcal{CS}, \qquad \frac{1-e^{-\lambda}}{\lambda}\in\mathcal{LS}\setminus\mathcal{CS},\qquad \sinh^{-1}(\lambda)\in\mathcal{CB}.
\end{align*}
Nevertheless, we particularly focus on the fractional framework, where the proposed approximation schemes, due to their flexibility in the parameters involved, unfold their full potential.

\section{Rational Krylov methods for Stieltjes and complete Bernstein functions}
\label{sec:RKM}
In view of Lemma \ref{lemma:fracstieltjes}, we use a rational Krylov method as a starting point to approximate matrix-vector products of the form $f(L)\mathbf{b}$, where $f\in\mathcal{SB}$ is a function of Stieltjes or complete Bernstein type. A brief introduction to RKMs is outlined in the following.

\subsection{The rational Krylov method}
Throughout what follows, let $L\in\R^{N\times N}$ label a diagonalizable positive definite matrix whose smallest and largest eigenvalues are given by $\lambda_{\min}$ and $\lambda_{\max}$, and $\Sigma := [\lambda_{\min},\lambda_{\max}]$ its spectral interval. Based on a selection of poles $\Xi = \{\xi_1,\dots,\xi_k\}\subset\mathbb{C}\setminus\Sigma$ we introduce the polynomial
\begin{equation*}
	q_{k}(\lambda) := \prod_{j=1}^k (\lambda - \xi_j) \in \mathcal P_{k},
\end{equation*}
where $\mathcal P_{k}$ denotes the algebraic polynomials of degree at most
$k$. We define the \emph{rational Krylov space} of $L$ and $\mathbf{b}$ associated to $\Xi$ by
\begin{align*}
	\mathcal Q^{\Xi}_{k+1}(L, \mathbf b) :=	q_k(L)^{-1} \mathcal K_{k+1}(L, \mathbf b) := q_k(L)^{-1}\operatorname{span} \{ \mathbf b, L \mathbf b, \ldots, L^k \mathbf b \}.
\end{align*}
If the poles are pairwise distinct, then $\mathcal Q^\Xi_{k+1}(L, \mathbf b) = \operatorname{span}\{\mathbf{b}, (I-\xi_1L)^{-1}\mathbf{b},\dots,(I-\xi_kL)^{-1}\mathbf{b}\}$ so that the computation of the rational Krylov space can be performed efficiently in parallel. The extraction of a proper surrogate $\mathbf{u}_{k+1}\approx f(L)\mathbf{b}$ relies on a matrix $V\in\C^{N\times(k+1)}$ whose columns form a basis of $\mathcal Q_{k+1}^{\Xi}(L, \mathbf b)$. The construction of such matrices typically involves Gram-Schmidt-type orthonormalization algorithms and can be found in the literature \cite{Ruhe1984,Guettel:PhD}. The desired approximation of the matrix-vector product is obtained via Rayleigh-Ritz extraction and reads
\begin{align}
	\label{eq:rkm}
	\mathbf u_{k+1} := V f(L_{k+1}) V^\dagger \mathbf b \in\mathcal{Q}^{\Xi}_{k+1}(L,\mathbf{b}),
	\qquad
	L_{k+1} := V^\dagger L V \in \mathbb \C^{(k+1)\times(k+1)},
\end{align}
where $V^\dagger\in\C^{(k+1)\times N}$ is the Moore-Penrose inverse of $V$ \cite{Geville2003}. Typically $k\ll N$ such that $f(L_{k+1})$ can be computed directly by diagonalization.

\begin{remark}
	In the finite element setting presented in Section \ref{sec:fracdiff}, we have $L = M^{-1}A$ with $M$ and $A$ as in \eqref{eq:femmatrices}. It this context, it is computationally convenient to choose $V$ as $M$-orthonormal basis of $\mathcal{Q}_{k+1}^\Xi(L,\mathbf{b})$ such that $V^\dagger = V^*M$ and thus $L_{k+1} = V^*AV$, where $V^*$ labels the adjoint of $V$. 
\end{remark}

The eigenvalues of $L_{k+1}$ are called \textit{rational Ritz values} of $L$ on $\mathcal{Q}^\Xi_{k+1}(L,\mathbf{b})$ and are independent of the particular basis $V$ \cite{Guettel:PhD}. They are contained in $\Sigma$ and play an essential role in Rayleigh-Ritz approximations. Provided a set of pairwise distinct nodes $\Lambda = \{\sigma_1,\dots,\sigma_{n+1}\}\subset\C$, we introduce the rational interpolant
\begin{align}
	\label{eq:ratint}
	r_{\Lambda,\Xi}^f \in\mathcal{R}_{n,k} := \{r = p_1 /p_2 : p_{1}\in\mathcal{P}_{n}, p_2\in\mathcal{P}_k\setminus\{0\} \}
\end{align}
as the unique rational function of degree $(n,k)$ with denominator $q_k$ that interpolates $f$ in $\Lambda$. The following two results establishes a close relation between the rational Ritz values and functions of type \eqref{eq:ratint}, which is instrumental in the analysis of RKMs. For a proof we refer to \cite[Lemma 4.6, Theorem 4.8]{Guettel:PhD}.
\begin{lemma}
	\label{lemma:exactness}
	Let $V$ be a basis of $\mathcal{Q}_{k+1}^{\Xi}(L,\mathbf{b})$, $L_{k+1} = V^\dagger LV$, and $r_k = p_k / q_k$ for some polynomial $p_k\in\mathcal{P}_k$. Then the rational Krylov approximation of $r_k(L)\mathbf{b}$ is exact, i.e.,
	\begin{align*}
	r_k(L) \mathbf{b} = Vr_k(L_{k+1})V^\dagger\mathbf{b}.
	\end{align*}
\end{lemma}
\begin{theorem}
	\label{thm:krylov_interpol}
	Let $V$ be a basis of $\mathcal{Q}_{k+1}^{\Xi}(L,\mathbf{b})$, $L_{k+1} = V^\dagger LV$, and $\mathbf{u}_{k+1} = Vf(L_{k+1})V^\dagger\mathbf{b}$. Then there holds
	\[
		\mathbf u_{k+1} = r(L) \mathbf b,
	\]
	where $r = p_k/q_k$, $p_k\in\mathcal{P}_k$, interpolates $f$ in the rational Ritz values $\Lambda = \{\mu_1,\dots,\mu_{k+1}\}$ of $L$ on $\mathcal{Q}_{k+1}^{\Xi}(L,\mathbf{b})$. If $(\mu)_{j=1}^{k+1}$ are pairwise distinct, then $r = r^f_{\Lambda,\Xi}$ as in \eqref{eq:ratint}.
\end{theorem}

The quality of approximation clearly depends on the rational Krylov space and the way it is extracted from it. As the following theorem shows, an extraction via \eqref{eq:rkm} yields a quasi-optimal surrogate. Its proof (see e.g., \cite[Theorem 4.10]{Guettel:PhD}) relies on Crouzeix's estimate \cite{Crouzeix2007},
\begin{align}
	\label{eq:Crouzeix}
	\norm{f(L)} \leq 2C \norm{f}_\Sigma,
\end{align}
where $C\leq 11.08$ is an absolute constant, $\norm{\cdot}$ the matrix norm induced by the Euclidean norm, which we label $\norm{\cdot}$ again, and $\norm{\cdot}_\Sigma$ the maximum norm on $\Sigma$. If $L$ is symmetric, the inequality holds with $C = 1$.

\begin{theorem}
	\label{thm:rkm_opt}
	Let $V$ be a basis of $\mathcal Q^{\Xi}_{k+1}(L,\mathbf{b})$, $L_{k+1} = V^\dagger L V$, and	$\mathbf u_{k+1} = V f(L_{k+1}) V^\dagger \mathbf b$. Then there holds
	\begin{align}
		\label{eq:quasioptimal}
		\norm{f(L) \mathbf b - \mathbf u_{k+1}} \le
		2 C \norm{\mathbf b} \min_{p \in \mathcal P_k} \norm{f - p/{q_k}}_{\Sigma}
	\end{align}
	with a constant $C \le 11.08$. If $L$ is self-adjoint, the result holds with $C=1$.
\end{theorem}

In many scenarios, Theorem \ref{thm:rkm_opt} provides a powerful tool to motivate a competitive selection of poles and allows to reduce the analysis to a rational approximation problem on $\Sigma$.  If $f = \ftau$, however, the rational approximation heavily depends on the problem-specific parameters and is thus unfeasible whenever $\boldsymbol{\tau}\mapsto\ftau(L)\mathbf{b}$ is queried for several instances of $\boldsymbol{\tau}$. To address this difficulty, we strive for a selection of poles which is independent of the parameters. The key ingredient for our approach relies on the observation that the approximability of a function $f\in\mathcal{SB}$ directly relates to the approximability of its respective kernel. Due to
\begin{align*}
	f(L)\mathbf{b} = \int_0^\infty \mu(\zeta)g(\zeta,L)\mathbf{b}\, d\zeta, \qquad  Vf(L_{k+1})V^\dagger\mathbf{b} = \int_0^\infty \mu(\zeta)Vg(\zeta,L_{k+1})V^\dagger\mathbf{b}\, d\zeta,
\end{align*}
we can bound the rational Krylov error
\begin{align}
	\begin{aligned}
	\label{eq:rkmbound}
		\norm{f(L)\mathbf{b} - Vf(L_{k+1})V^\dagger\mathbf{b}} &= \left\|\int_0^\infty \mu(\zeta)\left( g(\zeta,L)\mathbf{b} - Vg(\zeta, L_{k+1})V^\dagger\mathbf{b}\right) d\zeta \right\| \\
		&\leq \int_0^\infty \mu(\zeta)\norm{ g(\zeta,L)\mathbf{b} - Vg(\zeta, L_{k+1})V^\dagger\mathbf{b}} \, d\zeta,
	\end{aligned}
\end{align}
see also \cite{Robol2020}. It is thus justified to tailor the poles towards the corresponding kernel
\begin{align}
\label{eq:kernels}
	g(\zeta,\lambda)\in\{e^{-\zeta\lambda}, (\zeta+\lambda)^{-1}, \lambda(\zeta+\lambda)^{-1}\},
\end{align}  
instead of dealing with each function individually.

\subsection{Simultaneous approximability of the kernel functions}
\label{sec:kernelapprox}
In this section, we provide the essential preparations for the description of our analytical and numerical tools presented in Section \ref{sec:rkmfracdiff}. To this end, we show that the approximability of all three kernel functions \eqref{eq:kernels} directly relates to a particular rational approximation problem on the numerical range of $L$. As a starting point, we use a distribution of poles obtained by the generalized third Zolotar\"ev problem \cite{Zolotarev1877,Gonchar1969,Popov1988,Achieser1992,Wachspress2013}: Given two disjoint sets $\Phi$, $\Psi\subset\C$, we seek a rational function $r_k^*\in\mathcal{R}_{k,k}$ that minimizes the \textit{Zolotar\"ev number}
\begin{align}
	\label{eq:Zolotarev}
	Z_k(\Phi,\Psi) := \inf_{r_k\in\mathcal{R}_{k,k}}\frac{\sup\{|r_k(z)| : z\in\Phi\}}{\inf\{|r_k(z)| : z\in\Psi\}}.
\end{align}
A connection to RKMs is well-known and can be readily established. Whenever $\Lambda= \{\sigma_1,\dots,\sigma_n\}\subset\C$ is a set of pairwise distinct nodes, $\Theta\subset\C$, and $\zeta\in\Theta$, it can be readily verified that
\begin{align}
	\label{eq:resolventdiff}
	\frac{1}{\zeta+\lambda} - r_{\Lambda,\Xi}^f(\lambda) = \frac{1}{\zeta+\lambda}\frac{r_{\Lambda,\Xi}(\lambda)}{r_{\Lambda,\Xi}(-\zeta)}, \qquad r_{\Lambda,\Xi}(\lambda) := r_{\Lambda,\Xi}^0(\lambda) =  \frac{\prod_{j=1}^{n}(\lambda-\sigma_j)}{\prod_{j=1}^{k}(\lambda-\xi_j)},
\end{align}
where $r_{\Lambda,\Xi}^f\in\mathcal{R}_{n-1,k}$ is defined by \eqref{eq:ratint} with $f(\lambda) = (\zeta + \lambda)^{-1}$; see e.g., \cite{Guettel:PhD}. Together with quasi-optimality \eqref{eq:quasioptimal}, this yields for any matrix $V$ whose columns form a basis of $\mathcal{Q}^\Xi_{k+1}(L,\mathbf{b})$ 
\begin{align}
	\label{eq:resolvent_zolobound}
	\begin{aligned}
		\norm{(\zeta I+L)^{-1}\mathbf{b} - V(\zeta I_{k+1} + L_{k+1})^{-1}V^\dagger\mathbf{b}} &\leq 2C\norm{\mathbf{b}}\norm{(\zeta+\lambda)^{-1} - r_{\Lambda,\Xi}^f(\lambda)}_{\Sigma} \\
		&\leq 2C \norm{\mathbf{b}}\norm{(\zeta+\lambda)^{-1}}_{\Sigma} \frac{\norm{r_{\Lambda,\Xi}(\lambda)}_{\Sigma}}{|r_{\Lambda,\Xi}(-\zeta)|},
	\end{aligned}
\end{align}
where $I_{k+1}\in\R^{(k+1)\times(k+1)}$ denotes the unit matrix. If $|\Lambda| = |\Xi|$, the question of optimal poles for $f(\lambda) = (\zeta+\lambda)^{-1}$ directly relates to to the third Zolotar\"ev problem with $\Phi = \Sigma$ and $\Psi = -\Theta$. Whenever $f$ can be represented via
\begin{align}
	\label{eq:markov}
	f(\lambda) = \int_\Gamma \frac{\mu(\zeta)}{\zeta+\lambda}\, d\zeta
\end{align}
for some contour $\Gamma\subset\C$ winding around the spectrum of $L$, \eqref{eq:resolventdiff} can be generalized to
\begin{align}
	\label{eq:hermitewalsh}
	f(\lambda)- r_{\Lambda,\Xi}^f(\lambda) = \int_{\Gamma}\frac{r_{\Lambda,\Xi}(\lambda)}{r_{\Lambda,\Xi}(-\zeta)}\frac{\mu(\zeta)}{\zeta+\lambda}\, d\zeta.
\end{align}
Equation \eqref{eq:hermitewalsh} is commonly known as Hermite-Walsh formula for rational interpolants \cite{Walsh1935,Guettel2012,Guettel:PhD,Guettel2013} and suggests, in view of \eqref{eq:quasioptimal}, to consider \eqref{eq:Zolotarev} for $\Phi = \Sigma$ and $\Psi = -\Gamma$.

Explicit solutions of \eqref{eq:Zolotarev} for arbitrary configurations of $(\Phi, \Psi)$ are not available. For some particular geometries, the minimizer $r_k^*$ of \eqref{eq:Zolotarev} can be derived analytically in terms of elliptic functions. If, e.g., $\Phi = [a,b]\subset\R^+$ and $\Psi = -\Phi$, there holds
\begin{align}
	\label{eq:zolorealsol}
	r_k^*(x) = \frac{p_k^*(x)}{p_k^*(-x)},\qquad p_k^*(x) = \prod_{j=1}^{k}(x-\mathcal{Z}_j^{(k)}),
\end{align}
where
\begin{align}
\label{eq:zolopnts}
	\mathcal{Z}_{j}^{(k)} := b\operatorname{dn}\left(\frac{2(k-j)+1}{2k}K(\delta'),\delta'\right), \qquad \delta' := \sqrt{1-\delta^2},\qquad \delta := \frac{a}{b},
\end{align}
for all $j = 1,\dots,k$. Here, $\operatorname{dn}$ denotes the Jacobi elliptic function and $K$ the elliptic integral of first kind; see \cite[Section 16 \& 17]{Abramowitz1964}. Zolotar\"ev's number is known to be bounded by \cite{Beckermann2017}
\begin{align*}
	Z_k(\Phi,-\Phi) \leq 4e^{-2C^*k},
\end{align*}
where
\begin{align}
	\label{eq:Cstar}
	C^* := \frac{\pi K(\mu_1)}{4K(\mu)},\qquad \mu := \left(\frac{1-\sqrt{\delta}}{1+\sqrt{\delta}}\right)^2, \qquad \mu_1 := \sqrt{1-\mu^2}.
\end{align}
For small values of $\delta$, $C^*$ behaves like
\begin{align}
	\label{eq:asymptotic}
	C^* \approx \frac{\pi^2}{2\ln(4\delta^{-1})},
\end{align}
see e.g., \cite{Beckermann2017}. Since $r_k^*(-x) = 1 / r_k^*(x)$, we further have
\begin{align*}
	Z_k(\Phi,-\Phi) = \frac{\max\{|r^*_k(z)| : z\in\Phi\}}{\min\{|r^*_k(-z)| : z\in \Phi\}} = \frac{\max\{|r_k^*(z)| : z\in\Phi\}}{\min\{1/|r_k^*(z)| : z\in \Phi\}} = \max\{|r_k^*(z)|^2 : z\in\Phi\}.
\end{align*}
As the following theorem shows, the latter even minimizes the maximal deviation from zero among all rational functions of degree $(k,k)$. The reader is referred to \cite{Zolotarev1877,Lebedev2005,Wachspress2013,Beckermann2017} for a detailed exposition.
\begin{theorem}[Zolotar\"ev]
	\label{thm:zolotarev}
	Let $\Phi = [a,b]\subset\R^+$, $C^*$ as in \eqref{eq:Cstar}, and $\mathcal{Z} := \{-\mathcal{Z}_1^{(k)},\dots, -\mathcal{Z}_k^{(k)}\}$ with $\mathcal{Z}_j^{(k)}$ as in \eqref{eq:zolopnts}. Then there holds 
	\begin{align}
		\label{eq:zolominmax}
		\norm{r_\mathcal{Z}}_\Phi = \min_{\substack{\Xi\subset-\Phi \\ |\Xi| = k}} \norm{r_\Xi}_\Phi\leq 2e^{-C^*k},
	\end{align}
	where for any $\Xi = \{\xi_1,\dots,\xi_k\}$ we define
	\begin{align}
		\label{eq:product}
		r_\Xi(\lambda) := r_{-\Xi,\Xi}(\lambda) = \prod_{j=1}^k\frac{\lambda+\xi_j}{\lambda-\xi_j}.
	\end{align}
\end{theorem}

In several scenarios, the poles of \eqref{eq:zolorealsol} have proven themselves as excellent poles for RKMs, e.g., if $f(\lambda) = (\zeta+\lambda)^{-1}$, $\zeta\in\R^+$, is the resolvent \cite{Guettel2013,Guettel:PhD,Druskin2009}, $f(\lambda) = \exp(-t\lambda)$, $t\in\R^+$, the exponential \cite{Guettel2013,Druskin2009}, or $f(\lambda) = \lambda^{-\frac{1}{2}}$ the negative square root function \cite{Guettel2013}. A generalization of the latter for arbitrary powers contained in $(-1,1)$ can be found in \cite{DS:2019,DS:2020} and is based on the spectral interval of $L$. In its most general form, Zolotar\"ev's poles have been studied in \cite{Robol2020} in the case where $L$ is Hermitian and $f$ of Cauchy- or Laplace-Stieltjes type. In the latter case, the interval advocated therein, from which \eqref{eq:zolopnts} are sampled from, matches the spectral interval of $L$ and thus coincides with the choice proposed in \cite{DS:2020}. In the Cauchy-Stieltjes case, the authors of \cite{Robol2020} propose a more refined selection of $a$ and $b$ in \eqref{eq:zolopnts} which allows to improve the dictating constant in the exponential convergence result by a factor of two. The employment of two different pole distributions for Cauchy- and Laplace-Stieltjes functions, however, might not be fruitful in time-dependent fractional diffusion problems as it requires, in view of Lemma \ref{lemma:fracstieltjes}, to compute $\mathcal{Q}^\Xi_{k+1}(L,\mathbf{b})$ for $s+\alpha/2<1$ and $s+\alpha/2\ge 1$ separately. Instead, we prefer to sample the poles on $[-\lambda_{\max},-\lambda_{\min}]$ irrespective of the fractional parameters. We confirm that the analysis of \cite{Robol2020} directly translates to non-symmetric and positive definite matrices $L\in\R^{N\times N}$ and prove convergence in the case where $f$ is a complete Bernstein function. While we limit ourselves to real matrices only, it can be easily seen that our results immediately translate to the complex Hermitian case.

Starting with $g(\zeta,\lambda) = (\zeta + \lambda)^{-1}$, we show that for each kernel \eqref{eq:kernels} the selection of poles can be related to Zolotar\"ev's minimal deviation problem \eqref{eq:zolominmax}.
\begin{theorem}
	\label{thm:resolvent}
	Let $\zeta\in\R_0^+$, $C$ as in Theorem \ref{thm:rkm_opt}, $V$ a basis of $\mathcal{Q}^\Xi_{k+1}(L,\mathbf{b})$ with pairwise distinct poles $\Xi\subset-\Sigma$, $r_\Xi(\lambda)$ as in \eqref{eq:product}, and $L_{k+1} = V^\dagger LV$. Then there holds
	\begin{align}
		\label{eq:resolventbound}
		\norm{(\zeta I+L)^{-1}\mathbf{b} - V(\zeta I_{k+1} + L_{k+1})^{-1}V^\dagger\mathbf{b}} \leq \frac{2C}{\zeta + \lambda_{\min}}\norm{\mathbf{b}} \norm{r_\Xi}_{\Sigma} .
	\end{align}
\end{theorem}
\begin{proof}
	In line with \cite[Lemma 3]{DH2021}, we choose $\Lambda = -\Xi$ in \eqref{eq:resolvent_zolobound} to deduce 
	\begin{align*}
		\norm{(\zeta I+L)^{-1}\mathbf{b} - V(\zeta I_{k+1} + L_{k+1})V^\dagger\mathbf{b}} \leq \frac{2C}{\zeta + \lambda_{\min}} \norm{\mathbf{b}} \frac{\norm{r_\Xi}_{\Sigma}}{|r_\Xi(-\zeta)|}.
	\end{align*}
	The claim follows from $|r_\Xi(-\zeta)| \geq 1$ for all $\zeta\in\R_0^+$.
\end{proof}

The established interpolant in the proof of Theorem \ref{thm:resolvent} is a rational function of degree $(k-1,k)$ and thus does not exploit the full degree available in its enumerator. As a matter of fact, Theorem \ref{thm:resolvent} also holds if we replace $\mathcal{Q}^\Xi_{k+1}(L,\mathbf{b})$ by its $k$-dimensional subspace $\operatorname{span}\{(I-\xi_1L)^{-1}\mathbf{b},\dots,(I-\xi_kL)^{-1}\mathbf{b}\}$. For complete Bernstein functions, however, the presence of $\mathbf{b}$ is essential. To see this, we provide the following variant of Hermite-Walsh's formula.
\begin{lemma}
	\label{lemma:hermitewalsh2}
	Let $\zeta\in\C\setminus\{0\}$, $\Lambda = \{\sigma_1,\dots,\sigma_k\}\subset\C\setminus\{0\}$ pairwise distinct, $\widehat{\Lambda} = \Lambda\cup\{0\}$, $\Xi = \{\xi_1,\dots,\xi_k\}\subset\C$, $r_{\Lambda,\Xi}$ as in \eqref{eq:resolventdiff}, and  $r_{\widehat{\Lambda},\Xi}^f$ as in \eqref{eq:ratint} with $f(\lambda) = \lambda / (\zeta+\lambda)$. Then
	\begin{align*}
		\frac{\lambda}{\zeta+\lambda} - r_{\widehat{\Lambda},\Xi}^f(\lambda) = \frac{\lambda}{\zeta+\lambda}\frac{r_{\Lambda,\Xi}(\lambda)}{r_{\Lambda,\Xi}(-\zeta)}.
	\end{align*}
\end{lemma}
\begin{proof}
	It is easily verified that $\lambda / (\zeta+\lambda) - r_{\widehat{\Lambda},\Xi}^f(\lambda)$ is a rational function of degree $(k+1,k+1)$. Thanks to the interpolation property, we have
	\begin{align*}
		\frac{\lambda}{\zeta+\lambda} - r_{\widehat{\Lambda},\Xi}^f(\lambda) = \frac{c(\zeta)\lambda\prod_{j=1}^{k}(\lambda-\sigma_j)}{(\zeta+\lambda)\prod_{j=1}^{k}(\lambda-\xi_j)} = \frac{c(\zeta)\lambda}{\zeta+\lambda}r_{\Lambda,\Xi}(\lambda)
	\end{align*} 
	for some $\zeta$-dependent constant $c(\zeta)\in\C$. Multiplying both sides with $(\zeta+\lambda) / \lambda$ and setting $\lambda = -\zeta$ reveals $c(\zeta) = r_{\Lambda,\Xi}(-\zeta)^{-1}$, which completes the proof.
\end{proof}
Lemma \ref{lemma:hermitewalsh2} is instrumental for the proof of the following theorem, which establishes a connection between the approximation of the Bernstein kernel and Zolotar\"ev's minimal deviation problem.
\begin{theorem}
	\label{thm:completebernstein}
	Let $\zeta\in\R_0^+$, $C$ as in Theorem \ref{thm:rkm_opt}, $V$ a basis of $\mathcal{Q}^\Xi_{k+1}(L,\mathbf{b})$ with pairwise distinct poles $\Xi\subset-\Sigma$, $r_\Xi(\lambda)$ as in \eqref{eq:product}, and $L_{k+1} = V^\dagger LV$. Then there holds
	\begin{align*}
		\norm{L(\zeta I+L)^{-1}\mathbf{b} - VL_{k+1}(\zeta I_{k+1} + L_{k+1})^{-1}V^\dagger\mathbf{b}} \leq 2C\frac{\lambda_{\max}}{\zeta+\lambda_{\max}} \norm{\mathbf{b}} \norm{r_\Xi}_{\Sigma}. 
	\end{align*}
\end{theorem}
\begin{proof}
	If $\zeta = 0$, the claim trivially holds. Let therefore $\zeta>0$. In analogy to the proof of Theorem \ref{thm:resolvent}, we make use of \eqref{eq:quasioptimal} to observe
	\begin{align*}
		\norm{L(\zeta I+L)^{-1}\mathbf{b} - VL_{k+1}(\zeta I_{k+1} + L_{k+1})^{-1}V^\dagger\mathbf{b}} \leq 2C\norm{\mathbf{b}}\min_{p\in\mathcal{P}_k}\norm{\frac{\lambda}{\zeta+\lambda} - \frac{p(\lambda)}{q_k(\lambda)}}_{\Sigma}.
	\end{align*} 
	Let now $r_{\widehat{\Lambda},\Xi}^f$ be defined as in Lemma \ref{lemma:hermitewalsh2} with $\Lambda = -\Xi$. Then there holds
	\begin{align*}
		\norm{L(\zeta I+L)^{-1}\mathbf{b} - VL_{k+1}(\zeta I_{k+1} + L_{k+1})^{-1}V^\dagger\mathbf{b}} \leq 2C\frac{\lambda_{\max}}{\zeta+\lambda_{\max}} \norm{\mathbf{b}} \frac{\norm{r_\Xi}_{\Sigma}}{|r_{\Xi}(-\zeta)|}.
	\end{align*}
	The claim follows from the observation $|r_\Xi(-\zeta)|\geq 1$ for all $\zeta\in\R^+$.
\end{proof}

The treatment of the exponential kernel is more delicate. In line with \cite{Druskin2009} and \cite{Robol2020}, the idea is to bring $e^{-\zeta L}$ in the form of \eqref{eq:markov} via inverse Laplace transform, that is,
\begin{align}
	\label{eq:ExponentialLaplace}
	e^{-\zeta L} = \frac{1}{2\pi i}\int_{i\R} e^{\zeta z} (z I+ L)^{-1}\, dz.
\end{align}
This allows us to leverage our knowledge about the resolvent to gain insights into the approximability of the exponential function.
\begin{theorem}
	\label{thm:exponential}
	Let $\zeta\in\R_0^+$, $C$ as in Theorem \ref{thm:rkm_opt}, $V$ a basis of $\mathcal{Q}^\Xi_{k+1}(L,\mathbf{b})$ with pairwise distinct poles $\Xi\subset-\Sigma$, $r_\Xi(\lambda)$ as in \eqref{eq:product}, and $L_{k+1} = V^\dagger LV$. Then there holds
	\begin{align*}
		\norm{e^{-\zeta L}\mathbf{b} - Ve^{-\zeta L_{k+1}}V^\dagger\mathbf{b}} \leq 8C\gamma_{k} \norm{\mathbf{b}} \norm{r_\Xi}_{\Sigma} , \qquad \gamma_{k} := 2.23 + \frac{2}{\pi}\ln\left(4k\sqrt{\frac{\lambda_{\max}}{\lambda_{\min}\pi}}\right).
	\end{align*}
\end{theorem}
\begin{proof}
	The proof can be essentially taken from \cite[Theorem 2]{Robol2020} by incorporating \eqref{eq:Crouzeix} accordingly. For the sake of completeness, we provide its key ingredients in Appendix~\ref{sec:appendix}.
\end{proof}

\subsection{Approximability of Stieltjes and Bernstein functions}
\label{sec:fapprox}

Theorems \ref{thm:resolvent}, \ref{thm:completebernstein}, and \ref{thm:exponential} affirm that the approximability of the kernel functions \eqref{eq:kernels} can be related to Zolotar\"ev's minimal deviation problem. Due to \eqref{eq:rkmbound}, the same applies to arbitrary functions contained in $\mathcal{SB}$.  
\begin{theorem}
	\label{thm:zolomain}
	Let $L\in\R^{N\times N}$ be positive definite, $f\in\mathcal{SB}$, $C$ as in Theorem \ref{thm:rkm_opt}, $V$ a basis of $\mathcal{Q}^\Xi_{k+1}(L,\mathbf{b})$ with pairwise distinct poles $\Xi\subset -\Sigma$, and $\mathbf{u}_{k+1} = Vf(L_{k+1})V^\dagger\mathbf{b}$. Then
	\begin{align}
		\label{eq:certificate}
		\norm{f(L)\mathbf{b} - \mathbf{u}_{k+1}} \leq 2Cc_{k} \norm{\mathbf{b}} \norm{r_\Xi}_{\Sigma},
	\end{align}
	where $f(0^+) := \lim\limits_{\lambda\to 0^+}f(\lambda)$ and 
	\begin{align*}
		c_{k} := 
		\begin{cases}
			4\gamma_{k}f(0^+), &\text{if }f\in\mathcal{LS}, \\
			f(\lambda_{\min}), &\text{if }f\in\mathcal{CS}, \\
			f(\lambda_{\max}), &\text{if }f\in\mathcal{CB}.
		\end{cases}
	\end{align*}
\end{theorem}
\begin{proof}
	Starting with $f\in\mathcal{LS}$, we make use of \eqref{eq:rkmbound} and apply Theorem \ref{thm:exponential} to deduce
	\begin{align*}
		\norm{f(L)\mathbf{b} - \mathbf{u}_{k+1}} &\leq \int_0^\infty \mu_L(\zeta)\norm{ e^{-\zeta L}\mathbf{b} - Ve^{-\zeta L}V^\dagger\mathbf{b}} \, d\zeta\\
		&\leq 8C\gamma_{k} \norm{\mathbf{b}} \norm{r_\Xi}_{\Sigma} \int_0^\infty \mu_L(\zeta) \,d\zeta = 8C\gamma_{k}f(0^+) \norm{\mathbf{b}} \norm{r_\Xi}_{\Sigma}.
	\end{align*}
	If $f\in\mathcal{CS}$, Theorem \ref{thm:resolvent} reveals
	\begin{align*}
		\norm{f(L)\mathbf{b} - \mathbf{u}_{k+1}} &\leq \int_0^\infty \mu_C(\zeta)\norm{ (\zeta I+L)^{-1}\mathbf{b} - V(\zeta I_{k+1} + L_{k+1})^{-1}V^\dagger\mathbf{b}} \, d\zeta \\
		&\leq 2C \norm{\mathbf{b}} \norm{r_\Xi}_{\Sigma} \int_0^\infty \frac{\mu_C(\zeta)}{\zeta+\lambda_{\min}} \,d\zeta = 2Cf(\lambda_{\min})\norm{\mathbf{b}}\norm{r_\Xi}_{\Sigma}.
	\end{align*} 
	Finally, if $f\in\mathcal{CB}$, we conclude by means of Theorem \ref{thm:completebernstein}
	\begin{align*}
		\norm{f(L)\mathbf{b} - \mathbf{u}_{k+1}} &\leq \int_0^\infty \mu_B(\zeta)\norm{ L(\zeta I+L)^{-1}\mathbf{b} - VL_{k+1}(\zeta I_{k+1} + L_{k+1})^{-1}V^\dagger\mathbf{b}} \, d\zeta \\
		&\leq 2C \norm{\mathbf{b}} \norm{r_\Xi}_{\Sigma}  \int_0^\infty \frac{\lambda_{\max}}{\zeta+\lambda_{\max}}\mu_B(\zeta) \,d\zeta = 2Cf(\lambda_{\max}) \norm{\mathbf{b}} \norm{r_\Xi}_{\Sigma}.
	\end{align*}
\end{proof}
Invoking Theorem \ref{thm:zolotarev}, we immediately obtain exponential convergence rates when using the poles of \eqref{eq:zolorealsol} for building the rational Krylov space. For the moment, however, we leave the upper bound in \eqref{eq:certificate} as it is and discuss its convergence properties in Section \ref{sec:rkmfracdiff} in more detail.

\begin{remark}
	\label{rem:f0+}
	If $f\in\mathcal{LS}$, the approximation error is bounded by $f(0^+)$, which is only meaningful if $f(\lambda)$ is bounded when $\lambda\to0^+$. As indicated in \cite[Remark 4]{Robol2020}, this difficulty can be overcome by applying $\tilde{f}(\lambda) := f(\lambda + \eta)\in\mathcal{CS}$ for some $\eta\in(0,\lambda_{\min})$ to the shifted matrix $\tilde{L} := L - \eta I$ at the cost of slower convergence rates.
\end{remark}
If $f\in\mathcal{LS}$ admits a complex continuation to $\C\setminus\R^-$ and satisfies a certain decay condition on $i\R$, it is possible to replace the logarithmic factor $\gamma_k$ in Theorem \ref{thm:zolomain} with an absolute constant. In line with \eqref{eq:ExponentialLaplace}, the key idea is to bring, instead of its kernel, the function $f$ itself via Dunford-Taylor calculus in the form of \eqref{eq:markov} using the complex contour $\Gamma = i\R$. 

\begin{theorem}
	\label{thm:lsimprovement}
	If, under the assumptions of Theorem \ref{thm:zolomain}, $f\in\mathcal{LS}$ admits an extension to $\C\setminus\R^-$, which we call $f$ again, such that
	\begin{align*}
		c_f := \int_{i\R} \left|\frac{f(\zeta)}{\zeta+\lambda_{\min}}\right| \, d\zeta < \infty.
	\end{align*}
	Then there holds
	\begin{align}
		\label{eq:lsimprovement}
		\norm{f(L)\mathbf{b} - \mathbf{u}_{k+1}} \leq \frac{Cc_f}{\pi} \norm{\mathbf{b}} \norm{r_\Xi}_{\Sigma} .
	\end{align}
\end{theorem}
\begin{proof}
	In line with the proof of Theorem \ref{thm:resolvent}, we deduce from \eqref{eq:resolvent_zolobound} for any $\zeta\in i\R$
	\begin{align*}
		\norm{(\zeta I + L)^{-1}\mathbf{b} - V(\zeta I_{k+1} + L_{k+1})^{-1}V^\dagger\mathbf{b}} &\le \frac{2C}{|\zeta + \lambda_{\min}|} \norm{\mathbf{b}} \frac{\norm{r_\Xi}_{\Sigma}}{|r_\Xi(-\zeta)|} = \frac{2C}{|\zeta + \lambda_{\min}|} \norm{\mathbf{b}} \norm{r_\Xi}_{\Sigma},
	\end{align*}
	where the last equality is due to $|r_\Xi(-\zeta)| = 1$. We make use of the complex continuation of $f$ and choose the imaginary axis as the contour in Cauchy's formula to deduce
	\begin{align*}
		\norm{f(L)\mathbf b - \mathbf{u}_{k+1}} &\leq \frac{1}{2\pi} \int_{i\R}f(\zeta) \norm{(\zeta I + L)^{-1}\mathbf{b} - V(\zeta I_{k+1} + L_{k+1})^{-1}V^\dagger\mathbf{b}} \,d\zeta \\
		&\leq \frac{C}{\pi} \norm{\mathbf{b}} \norm{r_\Xi}_{\Sigma} \int_{i\R} \left|\frac{f(\zeta)}{\zeta + \lambda_{\min}}\right| \, d\zeta.
        \qedhere
	\end{align*}
\end{proof}

\begin{remark}
    The proof of Theorem \ref{thm:lsimprovement} is closely related to the results of \cite{Bailly2000} which show that the real parameters \eqref{eq:zolopnts} encoded in \eqref{eq:zolorealsol} satisfy at least necessary optimality conditions to minimize $Z_k(\Sigma,i\R)$. As of yet, it is not known whether \eqref{eq:zolorealsol} provides the true global minimum of $Z_k(\Sigma,i\R)$. As shown in \cite[Theorem 4.3]{Druskin2009}, however, any other rational function with possibly complex poles yields at most a two-fold decrease of the error.
\end{remark}

\section{The framework of fractional diffusion}
\label{sec:rkmfracdiff}
The following section represents the core of this paper. We compare and analyze several existing pole selection strategies, provide a certified error bound to assess their quality, and develop two new pole generation algorithms especially tailored to the framework of fractional diffusion. To this end, we apply our analytical findings from Section \ref{sec:RKM} to a few prototypical parametric functions $\ftau$ arising from both stationary and time-dependent problems. In particular, we deal with the question how the rational Krylov surrogates perform in the limit case and whether a uniform error bound can be established.

We start with the stationary case and consider $\ftau(\lambda) = \lambda^{s}$, $s\in(0,1)$, such that $\ftau\in\mathcal{CB}$. Provided a matrix of basis vectors $V$ of $\mathcal{Q}_{k+1}^\Xi(L,\mathbf{b})$, Theorem \ref{thm:zolomain} yields
\begin{align*}
	\norm{L^s\mathbf{b} - VL_{k+1}^sV^\dagger\mathbf{b}} \leq 2C\lambda_{\max}^s\norm{\mathbf{b}} \norm{r_\Xi}_{\Sigma}.
\end{align*}
Due to $\mathbf{b}\in\mathcal{Q}^\Xi_{k+1}(L,\mathbf{b})$ and Lemma \ref{lemma:exactness}, the rational Krylov approximation is exact for $s= 0$ and, assuming $k\geq 1$, $s = 1$, respectively. Provided $\lambda_{\max}\geq 1$, we thus obtain
\begin{align}
	\label{eq:xs}
	\sup_{s\in[0,1]}\norm{L^s\mathbf{b} - VL_{k+1}^sV^\dagger\mathbf{b}} \leq 2C\lambda_{\max} \norm{\mathbf{b}} \norm{r_\Xi}_{\Sigma} .
\end{align}
In a sense, this result shows that the discrete regularity assumptions in \cite{DS:2019} can be relaxed and, unlike prior results, bounds the rational Krylov error uniformly in $s\in[0,1]$.

If $s\in(-1,0)$, we have that $\ftau\in\mathcal{CS}$. The exactness property of the rational Krylov surrogate, however, does not hold for $s = -1$. In particular, the extremal function $f^{-1}(\lambda) = \lambda^{-1}$ is contained in $\mathcal{LS}\setminus\mathcal{CS}$. Thus, a feasible upper bound can only be obtained by Theorem \ref{thm:zolomain} if $s\in(-1,0)$. To circumvent this limitation, we directly bound the error using Theorem \ref{thm:resolvent} with $\zeta = 0$ to conclude
\begin{align*}
	\norm{L^{-1}\mathbf{b} - VL_{k+1}^{-1}V^\dagger\mathbf{b}} \le \frac{2C}{\lambda_{\min}} \norm{\mathbf{b}}\norm{r_{\Xi}}_{\Sigma}.
\end{align*}
All together, we thus obtain
\begin{align}
	\label{eq:x-s}
	\sup_{s\in[0,1]} \norm{L^{-s}\mathbf{b} - VL_{k+1}^{-s}V^\dagger\mathbf{b}} \le 2C\max\{1,\lambda_{\min}^{-1}\}\norm{\mathbf{b}} \norm{r_\Xi}_{\Sigma},
\end{align} 
which, as opposed to \cite{DS:2020} and \cite{DH2021}, bounds the error uniformly in $s$.

The treatment of time-dependent problems is more delicate. Unlike in the stationary case, the function $\ftau(\lambda) = E_{\alpha,\beta}(-t^\alpha\lambda^s)$ alternates membership between $\mathcal{CS}$ and $\mathcal{LS}\setminus\mathcal{CS}$ depending on its parameters (cf.~Lemma~\ref{lemma:fracstieltjes}). Arguing as in \cite[Example 4.3]{Novati2017}, we deduce from Podlubny's inequality \cite[Theorem 1.6]{Podlubny1999} the existence of a constant $c_\alpha>0$ such that 
\begin{align}
	\label{eq:calpha}
	\left|E_{\alpha,\beta}(-t^\alpha \lambda^s)\right| \leq \frac{c_\alpha}{1 + t^\alpha|\lambda|^s}, \rlap{$\qquad \lambda\in i\R,$}
\end{align}
for all $(\alpha, s)\in(0,1]^2$ with $s+\alpha < 2$. This allows us to show the following technical lemma, whose proof we provide in Appendix~\ref{sec:appendix}.
\begin{lemma}
	\label{lemma:calpha}
	Let $\alpha\in(0,1]$, $\beta\geq\alpha$, $t\in\R^+$, $s\in(0,1]$, $s+\alpha<2$, $f(\lambda) = E_{\alpha,\beta}(-t^\alpha \lambda^s)$, $c_\alpha$ as in \eqref{eq:calpha}, and $c_f$ as in Theorem \ref{thm:lsimprovement}. Then there holds
	\begin{align*}
		c_f \leq 2C_{\alpha,s, t} := 2c_\alpha\left(\lambda_{\min}^{-1} + s^{-1}\ln(1+t^{-\alpha})\right).
	\end{align*}
\end{lemma}
Provided the parameters meet the requirements of Lemma \ref{lemma:calpha}, Theorem \ref{thm:lsimprovement} immediately reveals
\begin{align}
	\label{eq:MLnonuniform}
	\norm{E_{\alpha,\beta}(-t^\alpha L^s)\mathbf{b} - VE_{\alpha,\beta}(-t^\alpha L_{k+1}^s)V^\dagger\mathbf{b}} \leq \frac{2CC_{\alpha,s,t}}{\pi}\norm{\mathbf{b}}\norm{r_\Xi}_{\Sigma}.
\end{align}
The constant $c_\alpha$ has been quantified in \cite[Proposition 4.6]{Turner2011} and suggests that $C_{\alpha,s,t}$ degenerates whenever $\alpha\to 0$, $s\to 0$, or $t\to 0$, such that \eqref{eq:MLnonuniform} does not allow for uniform boundedness in the parameters. In our experience, the constant $C_{\alpha,s,t}$  is rather pessimistic. Indeed, whenever $\alpha = 0$, $s = 0$, or $t = 0$, we have $\ftau(\lambda) = \ml(-t^\alpha, \lambda^s)\in\mathcal{LS}$, where $\ml$ is defined by \eqref{eq:mlextended}. At the cost of the additional logarithmic factor $\gamma_k$, we can thus consult Theorem \ref{thm:zolomain} to confirm uniform boundedness in the parameters,
\begin{align}
	\label{eq:MLuniform}
	\sup_{\boldsymbol{\tau}\in\Theta}\norm{\ml(-t^\alpha,L^s)\mathbf{b} - V\ml(-t^\alpha,L_{k+1}^s)V^\dagger\mathbf{b}} \leq 8C\gamma_k\norm{\mathbf{b}}\norm{r_\Xi}_{\Sigma}
\end{align}
with $\Theta := [0,1]\times[\alpha,\infty)\times\R_0^+\times[0,1]$.

\subsection{Pole selection strategies}
We now discuss, analyze, and compare a variety of available pole distributions $\Xi\in\{\mathcal{Z}, \mathcal{E}, \mathcal{G}, \mathcal{S}, \mathcal{B}_{\boldsymbol{\tau}}\}$, defined in the following sections, which are suitable to the study of fractional PDEs.

\subsubsection{Zolotar\"ev's poles - $\mathcal{Z}$}
\label{sec:zolo}
A selection of poles that perfectly fits our analytical setting is the one obtained by Zolotar\"ev's minimal deviation problem, that is, $\mathcal{Z} := \{-\mathcal{Z}_1^{(k)},\dots,-\mathcal{Z}_k^{(k)}\}$ with $\mathcal{Z}_j^{(k)}$ as in \eqref{eq:zolopnts} and $[a,b] := [\lambda_{\min}, \lambda_{\max}]$. These poles are computationally convenient as they only require the knowledge of the extremal eigenvalues of $L$. In particular, $\mathcal{Z}$ is parameter-independent and thus allows for an efficient querying of $\boldsymbol{\tau}\mapsto \mathbf{u}_{k+1}\approx\ftau(L)\mathbf{b}$ after the initial computation of the Krylov basis $V$. Thanks to the results given above, we are in a position to quantify their performance with the theoretical key result of this paper. 
\begin{theorem}
	\label{thm:fractionalconv}
	Let $L\in\R^{N\times N}$ be positive definite, $C$ as in Theorem \ref{thm:rkm_opt}, $\mathcal{Z}$ as in Theorem \ref{thm:zolotarev} with $a = \lambda_{\min}$ and $b = \lambda_{\max}$, $C^*$ as in \eqref{eq:Cstar}, $V$ a basis of $\mathcal{Q}_{k+1}^\mathcal{Z}(L,\mathbf{b})$, and $\mathbf{u}_{k+1} = V\ftau(L_{k+1})V^\dagger\mathbf{b}$.
	\begin{itemize}
		\item[1.] If $s\in(0,1)$ and $\ftau(\lambda) = \lambda^s$, then
		\begin{align*}
			\norm{L^s\mathbf{b} - \mathbf{u}_{k+1}} \leq 2C\lambda_{\max}^s e^{-C^*k}\norm{\mathbf{b}}.
		\end{align*} 
		In particular, if $k\ge 1$ and $\lambda_{\max}\geq 1$, we have
		\begin{align*}
			\sup_{s\in[0,1]}\norm{L^s\mathbf{b} - \mathbf{u}_{k+1}} \leq 2C\lambda_{\max} e^{-C^*k}\norm{\mathbf{b}}.
		\end{align*}
		\item [2.] If $\lambda_{\min}\geq 1$, $s\in(0,1)$ and $\ftau(\lambda) = \lambda^{-s}$, then
		\begin{align*}
			\norm{L^{-s}\mathbf{b} - \mathbf{u}_{k+1}} \leq 2C\lambda_{\min}^{-s} e^{-C^*k}\norm{\mathbf{b}}, \qquad \sup_{s\in[0,1]} \norm{L^{-s}\mathbf{b} - \mathbf{u}_{k+1}} \leq 2C e^{-C^*k}\norm{\mathbf{b}}.
		\end{align*} 
		\item[3.] If $\boldsymbol{\tau} := (\alpha, \beta, t, s)\in (0,1]\times[\alpha,\infty)\times\R^+\times(0,1]$, $s+\alpha<2$, and $C_{\alpha,s,t}$ as in Lemma \ref{lemma:calpha}, then
		\begin{align*}
			\norm{E_{\alpha,\beta}(-t^\alpha L^s)\mathbf{b} - \mathbf{u}_{k+1}} \leq \frac{2CC_{\alpha,s,t}}{\pi}e^{-C^*k}\norm{\mathbf{b}}.
		\end{align*}
		Let $\gamma_k$ as in Theorem \ref{thm:zolomain}, $\ftau(\lambda) = \ml(-t^\alpha,\lambda^s)$ defined by \eqref{eq:mlextended}, and $\Theta := [0,1]\times[\alpha,\infty)\times\R_0^+\times[0,1]$, then
		\begin{align*}
			\sup_{\boldsymbol{\tau}\in\Theta}\norm{\ml(-t^\alpha, L^s)\mathbf{b} - \mathbf{u}_{k+1}} \leq 8C\gamma_ke^{-C^*k}\norm{\mathbf{b}}.
		\end{align*}
	\end{itemize}
\end{theorem}
\begin{proof}
	This is a direct consequence of Theorem \ref{thm:zolotarev} and \eqref{eq:xs}, \eqref{eq:x-s}, \eqref{eq:MLnonuniform}, and \eqref{eq:MLuniform}, respectively.
\end{proof}
\begin{remark}
	If the exact extremal eigenvalues of $L$ are not available, one typically replaces them by some  numerical approximations $0<\lambda_L\leq \lambda_{\min} < \lambda_{\max}\leq\lambda_U$ to build $\mathcal{Z}$ thereupon. Thanks to \eqref{eq:asymptotic}, the performance of the surrogate $\mathbf{u}_{k+1}$ extracted from $\mathcal{Q}_{k+1}^\mathcal{Z}(L,\mathbf{b})$ deteriorates only logarithmically if the approximations of $\lambda_{\min}$ and $\lambda_{\max}$ become worse.
\end{remark}
A computational inconvenience of Zolotar\"ev's poles is the fact that they are not nested. In particular, all poles, and thus also the Krylov basis, must be recomputed whenever $k$ increases.

\subsubsection{Poles based on EDS - $\mathcal{E}$}
\label{sec:eds}
A nested counterpart to Zolotar\"ev's poles has been developed in \cite{Druskin2009}; see also \cite{Robol2020}. Using equidistributed sequences (EDS) as a building block, the idea is to construct an infinite sequence of poles which asymptotically yield the same convergence rates as the ones obtained by $\mathcal{Z}$. Starting with an arbitrary EDS $(s_j)_{j\in\N}\subset[0,1)$, such as $s_j := j\sqrt{2} - \lfloor j\sqrt{2}\rfloor$, one iteratively constructs the set of poles $\mathcal{E} := \{\xi_1,\dots,\xi_k\}$ via $\xi_j := -\lambda_{\max}\sqrt{t_j}$, where $t_j$ is the root of $g(t) - s_j$ defined by
\begin{align*}
	g(t) := \frac{1}{2M}\int_{\delta^2}^{t} \frac{dy}{\sqrt{(y-\delta^2)y(1-y)}}, \;\quad M:= \int_0^1 \frac{dy}{\sqrt{(1-y^2)(1-(1-\delta^2)y^2)}}, \;\quad \delta := \frac{\lambda_{\min}}{\lambda_{\max}}. 
\end{align*}
The roots $t_1,\dots,t_k$ can be computed numerically, e.g., by Newton's method. Even though the competitiveness of $\mathcal{E}$ has been confirmed for small values of $k$, rigorous upper bounds are lacking so far. Unlike $\mathcal{Z}$, however, $\mathcal{E}$ allows one to build $\mathcal{Q}_{k+1}^{\mathcal{E}}(L,\mathbf{b})$ incrementally and is thus convenient when adaptive accuracy control is required.

\subsubsection{Poles based on weak greedy algorithms - $\mathcal{G}$} 
\label{sec:weakgreedy}
The poles presented in Section \ref{sec:zolo} and \ref{sec:eds} are based on a scalar rational approximation problem obtained by the quasi-optimality property of RKMs. Recent results established in \cite{DH2021} suggest to directly approximate the matrix-kernels using so-called weak greedy algorithms. These algorithms are very popular in the reduced basis literature to alleviate the computational costs when evaluating solutions to parametric PDEs but have not attracted as much attention in RKMs yet. Starting with a basis $V=[\mathbf{v}_1]$ of $\mathcal{Q}_1(L,\mathbf{b}) = \operatorname{span}\{\mathbf{b}\}$, one variant is to inductively select the next pole according to a residual-based error estimator
\begin{align}
\label{eq:weakgreedy}
	\xi_{k+1} := -\operatornamewithlimits{arg\, max}_{\zeta\in\mathcal{T}} \norm{\mathbf{b} - V(\zeta I_{k+1} + L_{k+1})^{-1}V^\dagger\mathbf{b}},
\end{align}
where $\mathcal{T}\subset\C$ is the parameter domain in which we wish to approximate the resolvent. For computational purposes, one typically replaces $\mathcal{T}$ with a fine but finite training set $\mathcal{T}_{\textnormal{train}}\subset\mathcal{T}$. If $f\in\mathcal{CS}$, an intuitive choice for $\mathcal{T}$ is given by the positive real axis. Unfortunately, this requires discretizing an unbounded domain which is difficult to tackle numerically. The estimate \eqref{eq:resolventbound} shows, however, that the error is dictated by $\norm{r_{\Xi}}_{\Sigma}$, whose minimum is attained by Zolotar\"ev's poles. The latter are contained in the spectral interval of $L$ and thus justify the restriction to the bounded domain $\mathcal{T} = \Sigma$. Due to the close resemblance between the Cauchy-Stieltjes and complete Bernstein kernel, it is reasonable to believe that the same choice yields a competitive selection of poles for $f\in\mathcal{CB}$. Even though \eqref{eq:ExponentialLaplace} suggests to choose $\mathcal{T} = i\R$ if $f\in\mathcal{LS}$, the upper bound in \eqref{eq:resolventbound} justifies, similarly to the Cauchy-Stieltjes case, the selection of \textit{real} poles according to \eqref{eq:weakgreedy} using $\mathcal{T} = \Sigma$ for any $f\in\mathcal{SB}$. In light of these observations, we denote with $\mathcal{G}$ the set of poles obtained by \eqref{eq:weakgreedy} with $\mathcal{T} = \Sigma$.

While they are computationally more demanding than conventional pole selection strategies, weak greedy algorithms have the ability to incorporate spectral information about $L$ and the particular vector $\mathbf{b}$. Furthermore, it is known that any basis $V$ of $\mathcal{Q}_{k+1}^\mathcal{G}(L,\mathbf{b})$ satisfies
\begin{align}
\label{eq:greedyerror}
	\sup_{\zeta\in\Sigma}\norm{(\zeta I + L)^{-1}\mathbf{b} - V(\zeta I_{k+1} + L_{k+1})^{-1}V^\dagger\mathbf{b}}\leq ce^{-C^*k}
\end{align}
for some constant $c>0$ and $C^*$ as in Theorem \ref{thm:fractionalconv}; see \cite{Maday2002,DeVore2013,Bonito2020}. Even though we cannot provide a proof that \eqref{eq:greedyerror} implies exponential convergence of the surrogate $\mathbf{u}_{k+1} = V\ftau(L_{k+1})V^\dagger\mathbf{b}$ for $\ftau\in\mathcal{SB}$, our empirical investigations suggest the validity of such a conjecture.

\subsubsection{Spectral adaptive poles - $\mathcal{S}$}
\label{sec:spectraladap}
The pole sets $\mathcal{Z}$ and $\mathcal{E}$ strive for a minimization of the error uniformly over $\Sigma$. The latter might be a crude indicator of the true discretization error if the spectral density of the operator is nonuniform. To counteract this, a variety of spectral adaptive pole generation algorithms have been proposed \cite{Druskin2010, Druskin2011, Guettel2013b}. Their key ingredient is to quantify the error in terms of a rational function involving the poles and rational Ritz values of the RKM. In \cite{Druskin2010}, it is shown that the residual of the resolvent takes the form
\begin{align}
	\label{eq:residual}
	(\zeta I + L)^{-1}\mathbf{b} - V(\zeta I_{k+1} + L_{k+1})^{-1}V^\dagger\mathbf{b} = \frac{1}{r_{\Lambda,\Xi}(-\zeta)}r_{\Lambda,\Xi}(L)\mathbf{b}, 
\end{align} 
where $V$ is a basis of $\operatorname{span}\{(I-\xi_1L)^{-1}\mathbf{b},\dots,(I-\xi_kL)^{-1}\mathbf{b}\}$ with poles $\Xi = \{\xi_1,\dots,\xi_k\}$, $\Lambda = \{\mu_1,\dots,\mu_k\}$ are the rational Ritz values of $L$ on the span of $V$, and $r_{\Lambda,\Xi}$ is defined by \eqref{eq:resolventdiff}. Aiming for the minimization of \eqref{eq:residual}, the pole set $\mathcal{S} := \{\xi_1,\dots,\xi_k\}$ is defined inductively using $|r_{\Lambda,\mathcal{S}}(-\zeta)|^{-1}$ as the objective function over $\Sigma$, i.e., 
\begin{align*}
	\xi_{k+1} := \operatornamewithlimits{arg\,max}_{\lambda\in\Sigma} \prod_{j=1}^{k}\left|\frac{\lambda + \xi_j}{\lambda + \mu_j}\right|,
\end{align*}
where e.g., $\xi_1 := \lambda_{\min}$. The location of the extremum can be estimated using a sufficiently fine discrete training set $\mathcal{T}_{\textnormal{train}}\subset\Sigma$ to extract the poles from. In a sense, this algorithm also aims for an approximation of the resolvent function. Unlike $\mathcal{Z}$ and $\mathcal{E}$, however, the presence of $(\mu_j)_{j=1}^{k}$ allows for better adjustments towards the true discrete spectrum of $L$. Even though no analytical results are available, there is empirical evidence that such spectral methods outperform $\mathcal{Z}$ and $\mathcal{E}$ whenever the operator exhibits a nonuniform spectral density.

\subsubsection{BURA poles - $\mathcal{B}_{\boldmath{\tau}}$}
So far, all presented poles justify their selection based on the approximability of the integral kernels \eqref{eq:kernels}. Inherently different are RKMs based on the best uniform rational approximation (BURA) of $\ftau$ over $\Sigma$. The idea is to use the poles $\{\xi_1^{(k)},\dots,\xi_k^{(k)}\} =: \mathcal{B}_{\boldsymbol{\tau}}$ of the latter for building the rational Krylov space. For computational convenience, we limit ourselves to $\mathcal{B}_{\boldsymbol{\tau}}\subset\R$. In this case, it is known that the BURA exists and is unique if $\ftau$ is continuous in $\Sigma$; see \cite{Achieser1992}. Efficient and numerically stable algorithms for computing BURAs exist \cite{H:2020b} and can be applied as a black-box pole generator. A generic tool to quantify its quality is provided by Theorem \ref{thm:rkm_opt}. Based on the work of Stahl \cite{Stahl2003}, it has been shown (cf.~\cite{DH2021}) that
\begin{align*}
	\norm{L^{-s}\mathbf{b} - VL_{k+1}^{-s}V^\dagger\mathbf{b}} \leq C_s \lambda_{\min}^{-s}e^{-2\pi\sqrt{k s}}\norm{\mathbf{b}},
\end{align*}
where $C_s>0$ is some $s$-dependent constant and $V$ a basis of $\mathcal{Q}_{k+1}^{\mathcal{B}_{\boldsymbol{\tau}}}(L,\mathbf{b})$. Similar results hold for positive powers of the operator,
\begin{align*}
	\norm{L^s\mathbf{b} - VL_{k+1}^{s}V^\dagger\mathbf{b}} \leq C_s \lambda_{\max}^{s}e^{-2\pi\sqrt{k s}}\norm{\mathbf{b}},
\end{align*}
where $V$ is obtained using the poles based on the BURA of $\lambda^s$ in $\Sigma$; see \cite{Stahl2003,Pasciak2020}. Even though these results attest $\mathcal{B}_{\boldsymbol{\tau}}$ inferior convergence rates compared to those of Zolotar\"ev's poles, the numerical experiments in \cite{DH2021} suggest that the opposite is true. This is presumably due to the fact that the above estimates are valid for a spectral interval that is bounded only from one side, whereas in practice we approximate $f^\tau$ in the bounded spectral interval $\Sigma$, allowing for better rates. We are not aware of sharp error estimates for these best rational approximations in bounded intervals.

Rational approximations of the Mittag-Leffler function have been studied in, e.g., \cite{Starovoitov2007}. Explicit convergence rates of BURAs involving $\ftau(\lambda) = E_{\alpha,\beta}(-t^\alpha \lambda^s)$, however, are not available to the best of our knowledge. Apart from the fact that Krylov spaces obtained by $\mathcal{B}_{\boldsymbol{\tau}}$ are not nested, a major drawback is that they are parameter-dependent and thus unfeasible when it comes to querying the solution map $\boldsymbol{\tau}\mapsto \ftau(L)\mathbf{b}$ for multiple parameters.

\subsection{Error certification}
Nested pole sequences like the ones presented in Sections \ref{sec:eds}, \ref{sec:weakgreedy}, and \ref{sec:spectraladap} allow for an adaptive enrichment of the rational Krylov space until the sought accuracy is obtained. Typically, however, these schemes are difficult to analyze and guaranteed error bounds are not available. To mitigate this problem, we make use of our analysis to assess the quality of these poles. Assuming that $\norm{r_\Xi}_{\Sigma}$ is available, the right-hand sides of \eqref{eq:xs}, \eqref{eq:x-s}, \eqref{eq:MLnonuniform}, and \eqref{eq:MLuniform}, or more generally, \eqref{eq:certificate} and \eqref{eq:lsimprovement}, provide a computable upper bound of the rational Krylov error. It is clear that this bound can be crude, in particular for poles which do not weight the error uniformly over $\Sigma$. However, if $\Xi$ in a sense imitates Zolotar\"ev's optimality property, we can expect $\norm{r_\Xi}_{\Sigma}$ to provide an accurate predictor for the true error.

\subsubsection{Computation of $\norm{r_\Xi}_{\Sigma}$}
A conceptually straightforward approach to compute the maximal deviation of $r_\Xi$ is to evaluate its absolute value over a discrete training set $\mathcal{T}_\textnormal{train}\subset\Sigma$ and choose its maximizer as approximation for $\norm{r_\Xi}_{\Sigma}$. Somewhat cumbersomely, the training set must be provided by the user and needs to be chosen sufficiently fine to achieve a good approximation to the true global extremum. To counteract this, we present the following lemma which is instrumental in our computation of $\norm{r_\Xi}_{\Sigma}$.
\begin{lemma}
	\label{lemma:extrema}
	Let $\Xi=\{\xi_1,\dots,\xi_k\}\subset -\Sigma$ be a set of pairwise distinct poles with $\xi_k<\dots<\xi_1$. Then $r_\Xi'(\lambda)$ has exactly $k-1$ zeros $\lambda_1^*,\dots,\lambda_{k-1}^*$ in $\R^+$ that are local extrema of $r_\Xi(\lambda)$. There holds $-\xi_{j}<\lambda_j^*<-\xi_{j+1}$ for all $j = 1,\dots,k-1$ and
	\begin{align}
	\label{eq:deriv}
		r_\Xi'(\lambda) = -2\sum_{j = 1}^{k}\frac{\xi_j}{(\lambda-\xi_j)^2} r_{\Xi_j}(\lambda), \quad r_\Xi''(\lambda) = -2\sum_{j = 1}^k\frac{\xi_j}{(\lambda - \xi_j)^2}\left( r_{\Xi_j}'(\lambda) - \frac{2}{(\lambda-\xi_j)} r_{\Xi_j}(\lambda)\right),
	\end{align}   
	where $\Xi_j$ is the $k-1$ dimensional subset of $\Xi$ obtained by excluding $\xi_j$. 
\end{lemma} 
\begin{proof}
	The first part of the poof is a direct consequence of \cite[Proposition 2.3]{Druskin2010}. The identities in \eqref{eq:deriv} follow from the generalized product rule
	\begin{align*}
	\frac{d}{d\lambda}\left(\prod_{j=1}^k g_j(\lambda)\right) = \sum_{j=1}^k \frac{d g_j}{d\lambda}(\lambda)\prod_{\substack{i=1 \\ i\neq j}}^{k} g_i(\lambda),
	\end{align*}
	which holds for any collection of scalar and differentiable functions $(g_j)_{j=1}^k$.
\end{proof}
Thanks to Lemma \ref{lemma:extrema}, it suffices to compare the absolute values of $r_\Xi(\lambda_{\min})$ and $r_\Xi(\lambda_{\max})$ with those obtained by the $k-1$ local extrema $r_\Xi(\lambda_j^*)$ to determine the maximal deviation of $r_\Xi(\lambda)$ in $\Sigma = [\lambda_{\min},\lambda_{\max}]$. We propose to compute $(\lambda_{j}^*)_{j=1}^{k-1}$ by Newton's method on each subinterval utilizing the tools provided by Lemma \ref{lemma:extrema}. Feasible initial values can be obtained by evaluating $r_\Xi(\lambda)$ over a discrete training set $\mathcal{T}^j_\text{train}\subset(-\xi_j, -\xi_{j+1})$, $j = 1,\dots, k-1,$ of small cardinality. Provided that the initial value is sufficiently close to the true zero of $r_\Xi'$, Newton's algorithm is guaranteed to converge to the desired solution, as the following lemma shows.  

\begin{lemma}
	\label{lemma:newton}
    Let $\Xi = \{\xi_1,\dots,\xi_k\}\subset - \Sigma$ be a set of poles with $\xi_k < \dots < \xi_1$ and $\lambda_j^*$ the unique root of $r_\Xi'$ in $(-\xi_j, -\xi_{j+1})$ for some $j\in\{1,\dots,k-1\}$. Then there exists some $\varepsilon>0$ with $U_\varepsilon(\lambda_j^*)\subset(-\xi_j,-\xi_{j+1})$ such that for all initial values $\lambda_{j,0}^*\in U_\varepsilon(\lambda_j^*)$, Newton's method converges to $\lambda_j^*$.
\end{lemma}
\begin{proof}
	This is a direct consequence of Kantorovich's theorem \cite{Ezquerro2020}.
\end{proof}
We cannot quantify $\varepsilon$ in Lemma \ref{lemma:newton} to choose $\mathcal{T}^j_\textnormal{train}\subset (-\xi_j,-\xi_{j+1})$ sufficiently fine to guarantee $\lambda^*_{j,0}\in U_\varepsilon(\lambda_j^*)$ and thus convergence of Newton's method in $(-\xi_j,-\xi_{j+1})$. If it does converge, however, we can be certain that its limit yields the desired local extrema of $r_\Xi$.  This observation suggests to start with a coarse training set and, if the iteration does not converge after a few steps or leaves the interval $(-\xi_j,-\xi_{j+1})$, restart Newton with an initial value extracted from a refined training set. Even though this procedure guarantees convergence only after a finite amount of refinements, we observe that the iteration is fairly robust in the initial value and usually converges in a few steps if we choose $|\mathcal{T}_\textnormal{train}^j| = 20$ using equispaced points. We summarize this approach in the following algorithm.

\begin{algorithm}[Error Certification]
	\label{algo:certificate}
	Input: Spectral bounds $0<\lambda_L<\lambda_U$ with $\Sigma\subset[\lambda_L,\lambda_U]$ and a set of poles $\Xi = \{\xi_1,\dots,\xi_k\}\subset -\Sigma$ with $\xi_k < \dots < \xi_1$ and $k\geq2$.
	\begin{itemize}
		\item[1.] For $j = 1,\dots,k-1$ do
		\begin{itemize}
			\item[1.1] choose a discrete training set $\mathcal{T}_\textnormal{train}^j\subset(-\xi_j,-\xi_{j+1})$; e.g., equispaced points with $|\mathcal{T}_\textnormal{train}^j| = 20$   
			\item[1.2] set
			\begin{align*}
				\lambda_{j,0}^* = \operatornamewithlimits{arg\,max}_{\lambda\in\mathcal{T}_\textnormal{train}^j} |r_\Xi(\lambda)|,
			\end{align*}
			\item[1.3] apply Newton's method using \eqref{eq:deriv} and $\lambda_{j,0}^*$ as initial value to compute $\lambda_j^*$ as root of $r_\Xi'$ in $(-\xi_j,-\xi_{j+1})$. If the iteration does not converge in $(-\xi_j,-\xi_{j+1})$, refine $\mathcal{T}_\textnormal{train}^j$ and go back to $1.2$.
		\end{itemize}
		\item[2.] Set $\lambda_0^* = \lambda_L$, $\lambda_k^* = \lambda_U$, and
		\begin{align*}
			\Delta_\Xi = \operatornamewithlimits{max}_{j\in\{0,\dots,k\}} |r_\Xi(\lambda_j^*)|.
		\end{align*}
	\end{itemize}
	Output: Error certificate $\Delta_\Xi = \norm{r_\Xi}_{\Sigma}$ with $\norm{\ftau(L)\mathbf{b} - V\ftau(L_{k+1})V^\dagger\mathbf{b}}\leq 2Cc_k\Delta_\Xi\norm{\mathbf{b}}$ for any $f\in\mathcal{SB}$, $\mathbf{b}\in\R^N$, and $C$, $c_k$ as in Theorem \ref{thm:zolomain}.
\end{algorithm}
Unlike all other poles listed above, $\mathcal{B}_{\boldsymbol{\tau}}$ is not necessarily contained in $\Sigma$ such that Algorithm \ref{algo:certificate} cannot be consulted to obtain a meaningful error indicator. Nevertheless, we can apply Theorem \ref{thm:rkm_opt} to assess the quality of the rational Krylov surrogate obtained by the BURA $r_*^{\boldsymbol{\tau}}$ of $\ftau$ on $\Sigma$ by
\begin{align}
	\label{eq:buraindicator}
	\norm{\ftau(L)\mathbf{b} - \mathbf{u}_{k+1}} \leq 2C\Delta_{\mathcal{B}_{\boldsymbol{\tau}}}.
\end{align} 
where $\Delta_{\mathcal{B}_{\boldsymbol{\tau}}}:= \norm{\ftau - r_*^{\boldsymbol{\tau}}}_{\Sigma}$. Although analytically not available, $\Delta_{\mathcal{B}_{\boldsymbol{\tau}}}$ can be recovered numerically as a by-product while computing $\mathcal{B}_{\boldsymbol{\tau}}$ and is thus directly available. 

\subsection{Novel pole selection algorithms - $\mathcal{A}$ and $\mathcal{F}$}
In real-world scenarios, one is interested in identifying the smallest parameter $k\in\mathbb{N}$ such that the approximation error remains below a user-defined threshold. One possibility to achieve this is to adaptively construct a pole set $\Xi$, compute the error certificate using Algorithm \ref{algo:certificate}, and stop the procedure once the upper bound is smaller than the desired tolerance. The first two stages can be combined if one leverages the information provided by Algorithm \ref{algo:certificate} to build $\mathcal{Q}^\Xi_{k+1}(L,\mathbf{b})$ thereupon. We propose, using a similar concept as the one employed in \cite{Bagby1969, Druskin2010, Druskin2011, Guettel2013b}, two novel pole distributions $\Xi\in\{\mathcal{A},\mathcal{F}\}$, defined via Algorithms \ref{algo:adaptive} and \ref{algo:fulladaptive}, by greedily minimizing $\norm{r_\Xi}_{\Sigma}$.
\begin{algorithm}[Automatic pole selection algorithm - $\mathcal{A}$]
	\label{algo:adaptive}
	Input: Spectral bounds $0<\lambda_L<\lambda_U$ with $\Sigma\subset[\lambda_L,\lambda_U]$
	\begin{itemize}
		\item[1.] Set $\xi_1 = -\lambda_L, \xi_2 = -\lambda_U $, $\mathcal{A} = \{\xi_1, \xi_2\}$, and $k = 2$.
		\item[2.] Perform step 1 from Algorithm \ref{algo:certificate} to
            obtain $(\lambda_j^*)_{j=1}^{k-1}$.
		\item[3.] Set 
		\[
            \xi_{k+1} = -\operatornamewithlimits{arg\,max}_{\lambda \in\{\lambda_1^*,\dots,\lambda_{k-1}^*\}} |r_\Xi(\lambda)|
        \]
        and $\mathcal{A} = \mathcal{A}\cup\{\xi_{k+1}\}$. If the maximum is attained by several indices, choose one of them.
		\item[4.] Relabel the poles so that $\xi_{k+1} < \dots < \xi_1$.
		\item[5.] Set $k = k+1$ and go back to step $2$ until the desired accuracy is obtained. 
	\end{itemize}
\end{algorithm}

The advantage of this approach compared to extracting a maximizer of $|r_\mathcal{A}(\lambda)|$ from a discrete training set $\mathcal{T}_\text{train}\subset\Sigma$ is its ability to automatically detect the global maximum, without the risk of missing a critical value. Moreover, we can directly assess the maximal deviation of $r_{\mathcal{A}}$ over $\Sigma$ and thus obtain the certificate provided by Algorithm \ref{algo:certificate} as a by-product. We cannot provide a proof that our greedy algorithm generates an asymptotically optimal solution to Zolotar\"ev's deviation problem. Nevertheless, our empirical findings presented in Section~\ref{sec:numerics} indicate that the algorithm has this property.

In a sense, Algorithm \ref{algo:adaptive} is not fully automatic since it necessitates the availability of some rough spectral bounds. A heuristic approach to overcome this restriction is based on the observation that the eigenvalues of $L_{k+1}$ typically provide good approximations to the extremal eigenvalues of $L$. In light of the fact that the rational Ritz values are contained in $\Sigma$, an automated variant of Algorithm \ref{algo:adaptive} is obtained by iteratively adapting the underlying spectral interval based on the extremal eigenvalues of $L_{k+1}$. With this in mind, we present the fully automatic pole selection strategy for incrementally building the pole set $\mathcal{F}$.    

\begin{algorithm}[Fully automatic pole selection algorithm - $\mathcal{F}$]
	\label{algo:fulladaptive}
	$\,$
	\begin{itemize}
		\item[1.] Compute the Ritz values $\mu_1$ and $\mu_2$ of $L$ on the polynomial Krylov space $\mathcal{K}_2(L,\mathbf{b}) = \operatorname{span}\{\mathbf{b}, L\mathbf{b}\}$. Set $\xi_1 = -\mu_1, \xi_2 = -\mu_2$, $\mathcal{F} = \{\xi_1, \xi_2\}$, and $k = 2$. 
		\item[2.] Perform step 1 from Algorithm \ref{algo:certificate} to
            obtain $(\lambda_j^*)_{j=1}^{k-1}$.
		\item[3.] Set $\lambda_0^* = \mu_1$, $\lambda_k^* = \mu_k$,
		\[
            \xi_{k+1} = -\operatornamewithlimits{arg\,max}_{\lambda \in\{\lambda_0^*,\dots,\lambda_k^*\}} |r_\Xi(\lambda)|
        \]
        and $\mathcal{F} = \mathcal{F}\cup\{\xi_{k+1}\}$. If the maximum is attained by several indices, choose one of them.
		\item[4.] Relabel the poles so that $\xi_{k+1} < \dots < \xi_1$.
		\item[5.] Set $k = k+1$, compute the extremal Ritz values $\mu_1$ and $\mu_k$ of $L$ on $\mathcal{Q}_{k+1}^\mathcal{F}(L,\mathbf{b})$, and go back to step $2$ until the desired accuracy is obtained. 
	\end{itemize}
\end{algorithm}
We conclude this section with a systematic comparison of the presented pole configurations in Table \ref{table:poles}, incorporating (from top to bottom)  
\begin{itemize}
	\item[-] their dependence on $\boldsymbol{\tau}$, which is instrumental in the efficient querying of $\boldsymbol{\tau}\mapsto \ftau(L)\mathbf{b}$, 
	\item[-] their ability to construct a nested sequence of Krylov spaces, i.e., $\mathcal{Q}^\Xi_{k}(L,\mathbf{b})\subset\mathcal{Q}^\Xi_{k+1}(L,\mathbf{b})$,  
	\item[-] the required (user-provided) data to compute $\Xi$,
	\item[-] their ability to adapt to the spectral density of $L$,
	\item[-] their ability to incorporate information about the vector $\mathbf{b}$,
	\item[-] the presence of convergence results for the rational Krylov error,
	\item[-] the availability of an error certificate.
\end{itemize}
Zolotar\"ev's poles are the only poles for which explicit error bounds for arbitrary $\ftau\in\mathcal{SB}$ and $k\in\N$ are available. $\mathcal{E}$ allows for convergence results for all $\ftau\in\mathcal{SB}$ but solely in an asymptotic sense. Contrary, the analysis of RKMs based on $\mathcal{B}_{\boldsymbol{\tau}}$ depends on the particular function and is thus only available for some configurations of $\ftau$. While the dependence of $\mathcal{B}_{\boldsymbol{\tau}}$ on $\boldsymbol{\tau}$ is unfeasible when $\boldsymbol{\tau}\mapsto\ftau(L)\mathbf{b}$ is queried for several values of $\boldsymbol{\tau}$, the particular adjustment to the parameter typically yields superior convergence properties when the approximation of $\ftau(L)\mathbf{b}$ is desired for one fixed $\boldsymbol{\tau}$. Unlike for all other poles, Algorithm \ref{algo:certificate} cannot be consulted to assess the quality of $\mathcal{B}_{\boldsymbol{\tau}}$ and one needs to resort to \eqref{eq:buraindicator} to obtain a meaningful error indicator. The opposite is true for $\mathcal{F}$, however: since $\mathcal{F}$ seeks to avoid the explicit computation of $\lambda_{\min}$ and $\lambda_{\max}$, computing $\norm{r_\mathcal{F}}_{\Sigma}$ to control the rational Krylov error is only of limited use.
\begin{table}[ht]
	\centering
	\begin{tabular}[t]{lccccccc}
		\hline
		Pole set $\Xi$ &$\mathcal{Z}$ &$\mathcal{E}$ &$\mathcal{G}$ &$\mathcal{S}$ &$\mathcal{B}_{\boldsymbol{\tau}}$ &$\mathcal{A}$ &$\mathcal{F}$\\
		\hline
		$\boldsymbol{\tau}$-independence & $\checkmark$ & $\checkmark$ & $\checkmark$ & $\checkmark$ & $\times$ & $\checkmark$ & $\checkmark$\\
		nested & $\times$ & $\checkmark$ & $\checkmark$ & $\checkmark$ & $\times$ & $\checkmark$ & $\checkmark$\\
		user-provided data & $\hat{\Sigma}$ & $\hat{\Sigma}$ & $\hat{\Sigma}$, $\mathcal{T}_{\textnormal{train}}$ & $\hat{\Sigma}$,$ \mathcal{T}_{\textnormal{train}}$ & $\hat{\Sigma}$ & $\hat{\Sigma}$ & -\\
		spectral adaption & $\times$ & $\times$ & $\checkmark$ & $\checkmark$ & $\times$ & $\times$ & $\times$\\
		vector adaption & $\times$ & $\times$ & $\checkmark$ & $\times$ & $\times$ & $\times$ & $\times$ \\
		analysis & $\checkmark$ & asympt. & $\times$ & $\times$ & $\sim$ & $\times$ & $\times$\\
		certificate & $\checkmark$ & $\checkmark$ & $\checkmark$ & $\checkmark$ & $\checkmark$ & $\checkmark$ & $\sim$\\
		\hline
	\end{tabular}
	\caption{Properties of the poles $\mathcal{Z}$, $\mathcal{E}$, $\mathcal{G}$, $\mathcal{S}$, $\mathcal{B}_{\boldsymbol{\tau}}$, $\mathcal{A}$, $\mathcal{F}$, where $\hat{\Sigma} = [\lambda_L,\lambda_U]\supset\Sigma$ is an estimate for the spectral region and $\mathcal{T}_{\textnormal{train}}\subset\hat{\Sigma}$ a sufficiently fine training set.}
	\label{table:poles}
\end{table}

\section{Numerical examples}
\label{sec:numerics}
In the following experiments we underpin the effectiveness of the presented poles and error certificates and compare their performance in the course of a few prototypical space-time fractional diffusion problems. All numerical examples are implemented within the finite element library Netgen/NGSolve\footnote{\url{https://ngsolve.org/}} \cite{Netgen,NGSolve}. The implementation of Zolotar\"ev's poles is performed by means of the special function library from \texttt{Scipy}\footnote{\url{https://docs.scipy.org/doc/scipy/reference/special.html}}. The evaluation of the Mittag-Leffler function is performed using the \texttt{jscatter} software package\footnote{\url{https://pypi.org/project/jscatter/}}. The BURA poles are computed using the implementation of the BRASIL algorithm \cite{H:2020b} contained in the \texttt{baryrat}\footnote{\url{https://github.com/c-f-h/baryrat}} Python package.

Throughout this section, $V_h\subset H_0^1(\Omega)$ denotes a finite element space on the unit square $\Omega = (0,1)^2$ of dimension $N\in\N$ consisting of piecewise linear, globally continuous functions on a quasi-uniform triangular mesh of mesh-size $h = 0.01$. Setting $D=I\in\R^{N\times N}$ and $c = 0$ in \eqref{eq:operator}, we have $\Op = -\Delta$ such that $L = M^{-1}A$,  $M$ and $A$ as in \eqref{eq:femmatrices}, refers to a discrete approximation of the Laplacian. We choose $\mathbf{b}$ to be the coefficient vector of the $L_2$-orthogonal projection of the constant $1$ function onto $V_h$ in each of our experiments. We limit ourselves to time-dependent problems only, i.e., the evaluation of $\ml(-t^\alpha, L^s)\mathbf{b}$, $\ml$ defined by \eqref{eq:mlextended},  with $\beta = 1$. The interested reader is referred to \cite{DS:2019, DS:2020, DH2021} for a detailed investigation of the stationary case.

\subsection{Parameter study}
The goal of this section is to illuminate the impact of the parameters on the rational Krylov approximation. To this end, we introduce the discrete $L_2$-error
\begin{align}
	\label{eq:l2error}
	E(k, \Xi, \boldsymbol{\tau}) := E(k,\Xi, \alpha, t, s) := \norm{\ftau(L)\mathbf{b} - V\ftau(L_{k+1})V^\dagger\mathbf{b}}_{M},
\end{align}
where $\ftau(\lambda) = \textrm{e}_{\alpha,1}(-t^\alpha,\lambda^s)$ is defined by \eqref{eq:mlextended} and $V$ a $M$-orthonormal basis of $\mathcal{Q}_{k+1}^\Xi(L,\mathbf{b})$. For now, we assume $\Xi = \mathcal{Z}$ to be the Zolotar\"ev poles on $[\lambda_L, \lambda_U] := [19, 348475]\supset\Sigma$ obtained by a numerical approximation of $\lambda_{\min}$ and $\lambda_{\max}$. 

The evolution of the error \eqref{eq:l2error} in $t$ is depicted in Figure \ref{fig:evolution} for $k = 7$, $s = 0.75$, and different configurations of $\alpha$. In the case of $\alpha = 0.25$, we are, according to Lemma \ref{lemma:fracstieltjes}, in the Cauchy-Stieltjes regime. Theorem \ref{thm:zolomain} predicts that the error decays like $\mathcal{O}(\textrm{e}_{\alpha,1}(-t^\alpha, \lambda_L^{0.75}))$ when $t\to\infty$, which is precisely what we observe in Figure \ref{fig:evolution}. Since $\textrm{e}_{\alpha,1}(-t^\alpha, \lambda^s)\in\mathcal{LS}\setminus\mathcal{CS}$ whenever $s+\frac{\alpha}{2}\geq 1$, we cannot confirm analytically that the error satisfies such a property if $\alpha\in\{0.5,0.75\}$. However, our numerical experiments suggest that \eqref{eq:l2error} can be bounded using $\ftau(\lambda_L)$ regardless of $(\alpha,s)\in[0,1]^2$. 

Using the same values for the fractional parameters as before, we study the limiting behaviour of \eqref{eq:l2error} for $t \to 0$ in Figure \ref{fig:tlimit}. In accordance with the theory, the error remains uniformly bounded in $t$ for all configurations of $\alpha$ and $s$. Since $\textrm{e}_{\alpha,1}(-t^\alpha,\lambda^s) \equiv 1$ for $t = 0$, in which case the rational Krylov approximation is exact, the error decreases as $t$ approaches zero. The observed rate of convergence is proportional to $t^{\alpha}$. As the example shows, however, the preasymptotic regime might be very large and heavily depends on the parameter $\alpha$.  
\begin{figure}[ht]
	\begin{minipage}[t]{0.485\linewidth}
		\centering							
		\includegraphics[width=\textwidth]{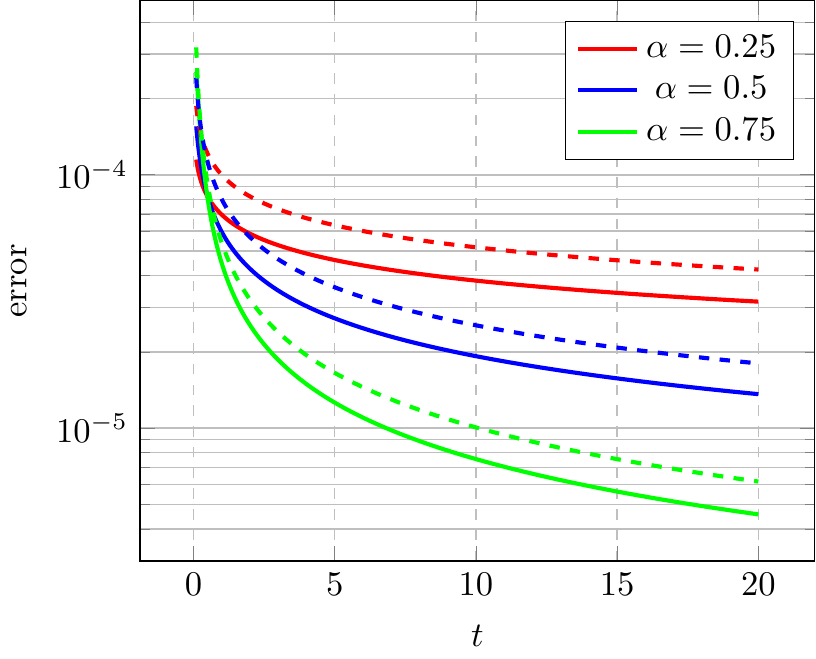}
		\caption{Error $E(7,\mathcal{Z}, \alpha, t, 0.75)$ (solid) for $\alpha = 0.25, 0.5, 0.75$ and $t\in[0.1, 20]$. The reference lines (dashed) are proportional to $\textrm{e}_{\alpha,1}(-t^\alpha, \lambda_L^{0.75})$ according to the respective value of $\alpha$.}
		\label{fig:evolution}
	\end{minipage}
	\hspace{0.1cm}
	\begin{minipage}[t]{0.485\linewidth}
		\centering							
		\includegraphics[width=\textwidth]{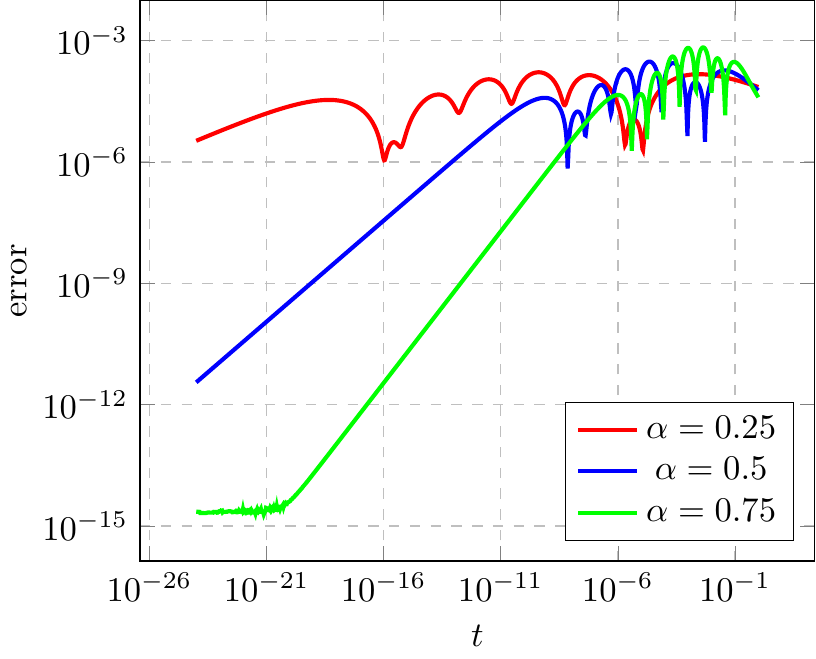}
		\caption{Error $E(7,\mathcal{Z}, \alpha, t, 0.75)$ for $\alpha = 0.25, 0.5, 0.75$ and $t\in[10^{-24},1]$.}
		\label{fig:tlimit}
	\end{minipage}
\end{figure}

To understand the sensitivity of the error with respect to the fractional parameters, we fix $t = 1.5$ to illustrate the spatial error as function in $\alpha$ and $s$ in Figure \ref{fig:parameter}. The quantity $E(7,\mathcal{Z}, \alpha, 1.5, s)$ is evaluated over a discrete parameter grid contained in $[0,1]^2$. Whenever the Euclidean norm of $(\alpha, s)\in\R^2$ is close to $\sqrt{2}$ or $s \ll 1$, we see that the error is small compared to other configurations of the fractional parameters. The former, in a sense, underpins our observations from Figure \ref{fig:evolution} that the error is proportional to $\ftau(\lambda_L)$ irrespectively of $(\alpha, s)\in[0,1]^2$ and thus, assuming $t\geq 1$, decreases whenever $\alpha$ or $s$ approach $1$. 
Contrary to the proportionality to $\ftau(\lambda_L)$, the quality of the surrogate improves also for small values of the spatial fractional parameter. This can be seen as a local approximation effect since $\textrm{e}_{\alpha,1}(-t^\alpha,\lambda^s)\equiv\text{const.}$ if $s = 0$, in which case the rational Krylov approximation is exact. In this regime, the error appears to be less prone to increasing values of $\alpha$.
\begin{figure}[ht]
	\begin{minipage}[t]{0.485\linewidth}
		\centering							
		\includegraphics[width=\textwidth]{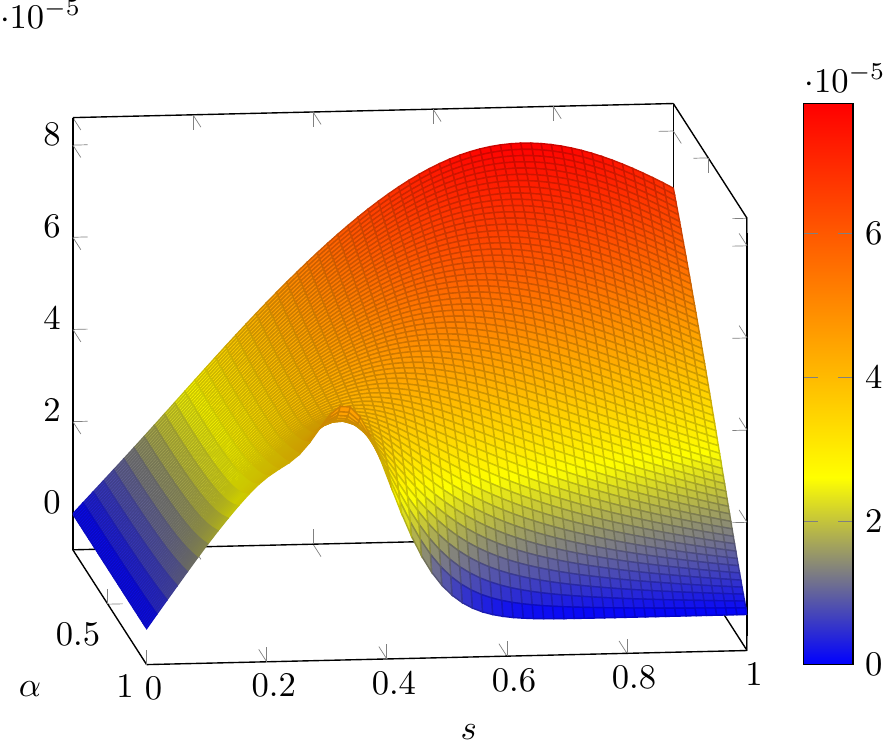}
		\caption{Error $E(7,\mathcal{Z}, \alpha, 1.5, s)$ for $(\alpha,s)\in[0,1]^2$.}
		\label{fig:parameter}
	\end{minipage}
	\hspace{0.1cm}
	\begin{minipage}[t]{0.485\linewidth}
		\centering							
		\includegraphics[width=\textwidth]{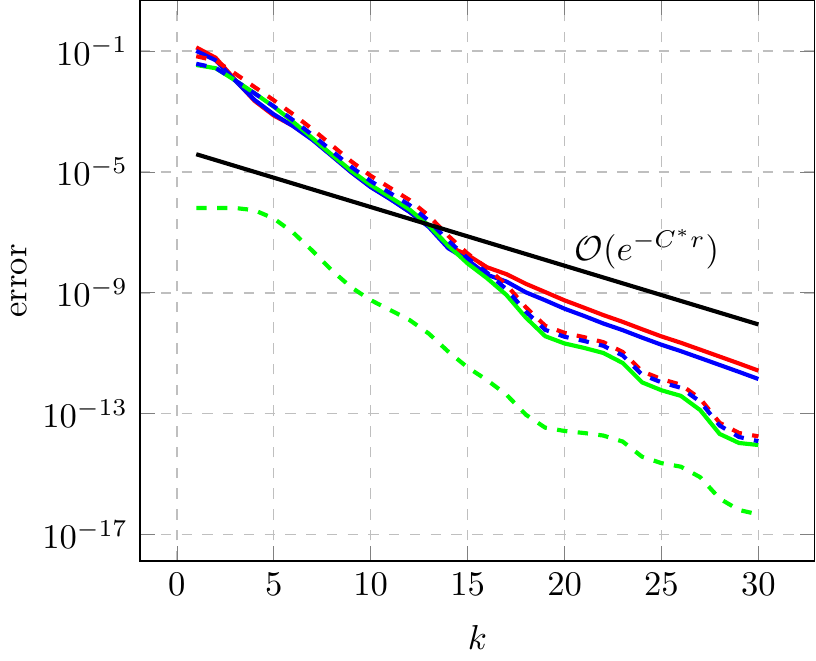}
		\caption{Error $E(k,\mathcal{Z}, \alpha, 1.5, s)$ for $s = 0.25$ (solid) and $s = 0.75$ (dashed) and $\alpha = 0.1, 0.5, 1$ in red, blue, and green, respectively.}
	\label{fig:zolo}
	\end{minipage}
\end{figure}

\subsection{Convergence study}
We now focus on a numerical confirmation of the uniform convergence of \eqref{eq:l2error} when using $\Xi = \mathcal{Z}$ and compare the latter with the poles presented in Section \ref{sec:rkmfracdiff}. Starting with Zolotar\"ev's poles, we fix $t = 1.5$ to monitor \eqref{eq:l2error} as a function in $k$ for various $\alpha$ and $s$ in Figure \ref{fig:zolo}. Lemma \ref{lemma:fracstieltjes} shows that $\ftau\in\mathcal{LS}\setminus\mathcal{CS}$ if and only if $(\alpha, s)\in\{(0.5,0.75),(1,0.75)\}$. Unexpectedly, the quality of the approximation does not deteriorate in these two cases compared to the other configurations of the fractional parameters. Instead, the experiment affirms our numerical observations from Figure \ref{fig:parameter} that the error decreases whenever $\alpha$ or $s$ approach $1$. In particular, the best result among all tested configurations is obtained by $(\alpha,s) = (1,0.75)$. Its approximation slightly outperforms the expected convergence rate of order $\mathcal{O}(e^{-C^*k})$, where $C^*\approx 0.45$. In general, we observe that for increasing values of the fractional parameters the preasymptotic regime becomes larger, such that for $(\alpha,s)$ close to $(1,1)$ the surrogate frequently reaches machine precision before the expected decay rate becomes visible.  

The performance of each RKM is deeply connected with the choice of poles. Therefore, we report the error \eqref{eq:l2error} for different configurations of $\Xi\in\{\mathcal{Z}, \mathcal{E}, \mathcal{G}, \mathcal{S}, \mathcal{B}_{\boldsymbol{\tau}}, \mathcal{A}, \mathcal{F}\}$, defined in Section \ref{sec:rkmfracdiff}, with respect to $[\lambda_L,\lambda_U]$ in Figure \ref{fig:poles1} and \ref{fig:poles2} using $(\alpha,s) = (0.5,0.5)$ and $(\alpha,s) = (1,0.5)$, respectively. In any case, we see that all numerical schemes satisfy exponential convergence rates. The BURA poles provide the best approximation among all tested configurations irrespective of the parameters involved. In order to achieve the same accuracy as $\mathcal{B}_{\boldsymbol{\tau}}$, roughly twice as many iterations are required by all other poles, which perform qualitatively similar compared to each other. This is due to the fact that the eigenvalue distribution of the discrete Laplacian is roughly uniform such that spectral adaptive poles do not significantly differ from poles which measure the error uniformly over $\Sigma$. In particular, we observe that the rational Krylov errors of $\mathcal{Z}$, $\mathcal{E}$, and $\mathcal{A}$ are almost coincident. In a sense, this is reasonable since all of these poles aim directly for a minimization of $\norm{r_\Xi}_{\Sigma}$. 

Theorem \ref{thm:fractionalconv} predicts exponential convergence rates of order $\mathcal{O}(e^{-C^*k})$ whenever $\Xi = \mathcal{Z}$. Only in an asymptotic sense, the same is known to hold for $\mathcal{E}$ but appears to be accurate already for small values of $k$. Similar results can be expected for $\mathcal{A}$ and $\mathcal{F}$, assuming that greedily minimizing $\norm{r_\Xi}_{\Sigma}$ provides an asymptotically optimal solutions to Zolotar\"ev's minimal deviation problem. For $\Xi = \mathcal{G}$, it is just as reasonable to expect an error proportional to $e^{-C^*k}$ since \eqref{eq:greedyerror} shows that the resolvent can be approximated with the aforementioned rate. Contrary to these predictions, however, either of these pole configurations decays slightly faster than $e^{-C^*k}$ and reaches machine precision before the expected rate of decay can be identified. Even though our numerical findings attest $\mathcal{B}_{\boldsymbol{\tau}}$ the best performance among all tested configurations, it becomes worthwhile to use one of the parameter independent pole selection algorithms whenever solutions for several values of $\boldsymbol{\tau}$ are required.
\begin{figure}[ht]
	\begin{minipage}[t]{0.485\linewidth}
		\centering							
		\includegraphics[width=\textwidth]{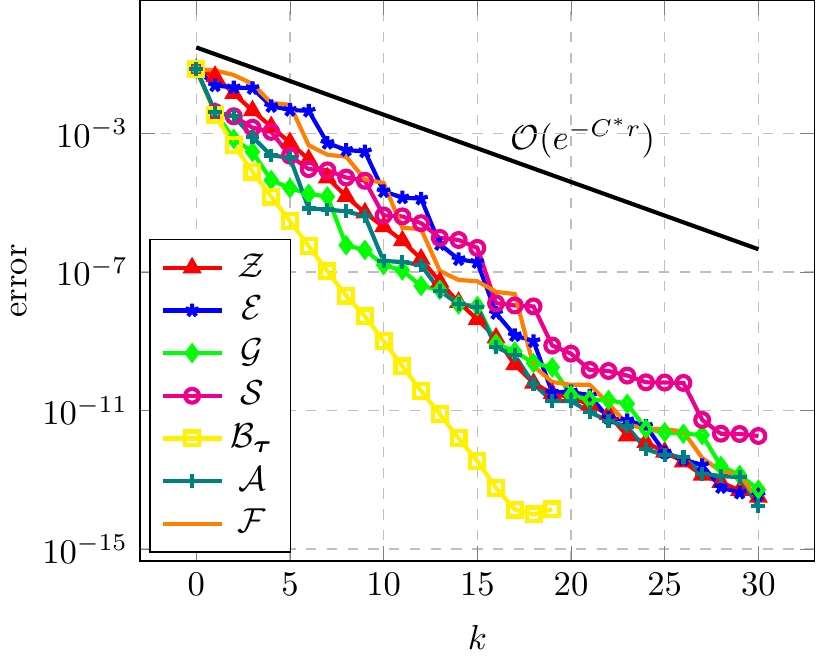}
		\caption{Error $E(k,\Xi, 0.5, 1.5, 0.5)$ with $\Xi\in\{\mathcal{Z}, \mathcal{E}, \mathcal{G}, \mathcal{S}, \mathcal{B}_{\boldsymbol{\tau}}, \mathcal{A}, \mathcal{F}\}$.}
		\label{fig:poles1}
	\end{minipage}
	\hspace{0.1cm}
	\begin{minipage}[t]{0.485\linewidth}
		\centering							
		\includegraphics[width=\textwidth]{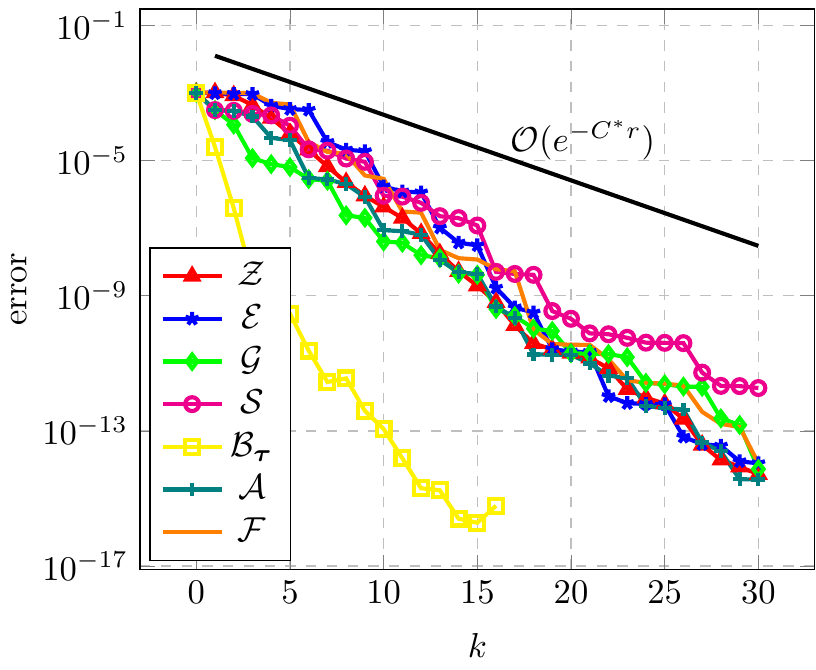}
		\caption{Error $E(k,\Xi, 1, 1.5, 0.5)$ with $\Xi\in\{\mathcal{Z}, \mathcal{E}, \mathcal{G}, \mathcal{S}, \mathcal{B}_{\boldsymbol{\tau}}, \mathcal{A}, \mathcal{F}\}$.}
		\label{fig:poles2}
	\end{minipage}
\end{figure}
\begin{remark}
	In light of Lemma \ref{lemma:fracstieltjes}, Figure \ref{fig:poles1} and \ref{fig:poles2} solely display the effectiveness of the discussed pole distributions for Laplace- and Cauchy-Stieltjes functions. For brevity, we omit the presentation of their performance in the context of complete Bernstein functions, such as $\ftau(\lambda) = \lambda^s$. We note, however, that in this scenario the poles perform qualitatively similar compared to $\ftau(\lambda) = \textrm{e}_{\alpha,1}(-t^\alpha,\lambda^s)$ and refer to \cite{DS:2019} for an investigation based on $\Xi = \mathcal{Z}$.
\end{remark}

We conclude this investigation by studying the performance of the developed error certificate as a predictor for \eqref{eq:l2error}. On the basis of Algorithm \ref{algo:certificate}, we illustrate the quantities $\norm{r_\Xi}_{\Sigma}$ in Figure \ref{fig:certificate} for $\Xi\in\{\mathcal{Z}, \mathcal{E}, \mathcal{G}, \mathcal{S}, \mathcal{A}, \mathcal{F}\}$. Since the certificate \eqref{eq:buraindicator} depends on the particular parameters, we omit $\Xi = \mathcal{B}_{\boldsymbol{\tau}}$ in our discussion. In accordance with the analysis, $r_\mathcal{Z}$ yields the least deviation from zero for all values of $k$ and decays proportionally to $e^{-C^*k}$. Apart from $\mathcal{Z}$, the poles obtained by Algorithm \ref{algo:adaptive} are among the most competitive parameters when it comes to the minimization of the certificate. Initially, $\norm{r_\mathcal{E}}_{\Sigma}$ is larger than the corresponding value for several other pole distributions, but improves as $k$ increases. In a sense, this observation confirms the fact that $\mathcal{E}$ satisfies Zolotar\"ev's minimal deviation property only asymptotically. Also $\norm{r_{\mathcal{G}}}_{\Sigma}$ provides a reasonable predictor for the error. For both $\mathcal{S}$ and $\mathcal{F}$, the certificate appears to be rather pessimistic. In the computation of the latter, the approximation of $\Sigma$ by the extremal rational Ritz values might be crude for small $k$, yielding the maximal deviation of $r_\mathcal{F}$ to be be small in the spectral interval of $L_{k+1}$ but large for values close to the extremal eigenvalues of $L$. The poor performance of $\norm{r_\mathcal{F}}_{\Sigma}$ as an error predictor, however, has no practical relevance since one can resort to $\Xi = \mathcal{A}$ whenever information about the spectral region of $L$ is available.
\begin{figure}[ht]
		\centering							
		\includegraphics[width=0.5\textwidth]{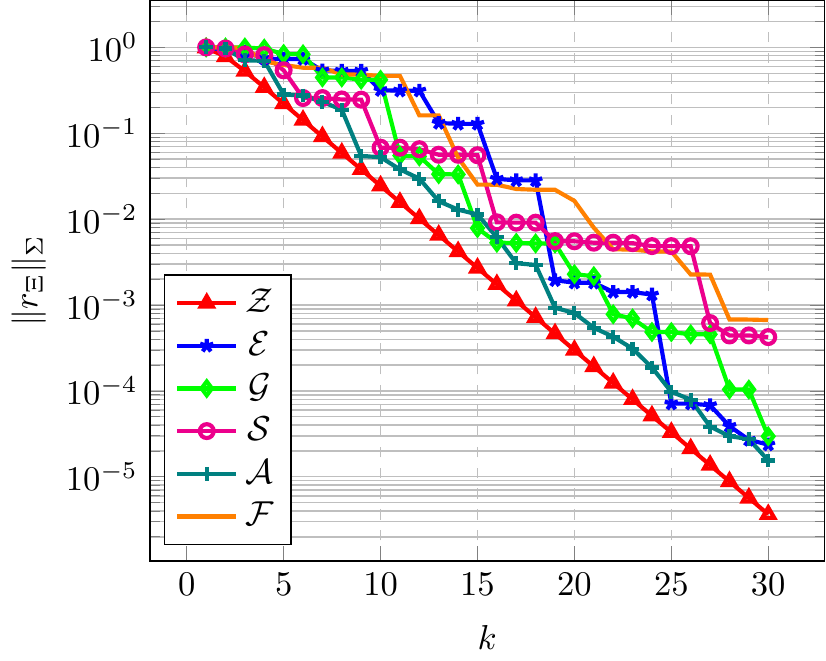}
		\caption{Error certificate $\norm{r_\Xi}_{\Sigma}$ for different poles $\Xi\in\{\mathcal{Z}, \mathcal{E}, \mathcal{G}, \mathcal{S}, \mathcal{A}, \mathcal{F}\}$ in dependence of its cardinality $|\Xi| = k$.}
		\label{fig:certificate}
\end{figure}

\appendix

\section{Appendix}
\label{sec:appendix}

\begin{proof}[Proof of Theorem \ref{thm:exponential}]
	The rational Krylov approximation \eqref{eq:rkm} is independent of the basis \cite[Lemma 3.3]{Guettel:PhD}. W.l.o.g.~we can therefore assume that $V$ is a matrix whose columns form an orthonormal basis of $\mathcal{Q}_{k+1}^\Xi(L,\mathbf{b})$. Using \eqref{eq:ExponentialLaplace}, we rewrite
	\begin{align*}
		e^{-\zeta L}\mathbf{b} - Ve^{-\zeta L_{k+1}}V^\dagger\mathbf{b} = &\frac{1}{2\pi i}\int_{i\R} e^{\zeta z}\left((z I+L)^{-1}\mathbf{b} - V(z I_{k+1} + L_{k+1})^{-1}V^\dagger\mathbf{b}\right)\,dz.
	\end{align*}
    Let $r_{-\Xi,\Xi}^f$ be defined as in \eqref{eq:ratint} with $f(\lambda) = (z+\lambda)^{-1}$. Due to Lemma \ref{lemma:exactness}, we have $r^f_{-\Xi,\Xi}(L)\mathbf{b} = Vr^f_{-\Xi,\Xi}(L_{k+1})V^\dagger\mathbf{b}$. Subtracting $e^{\zeta z}r^f_{-\Xi,\Xi}(L)\mathbf{b}$ and adding $e^{\zeta z}Vr^f_{-\Xi,\Xi}(L_{k+1})V^\dagger\mathbf{b}$ inside the integral combined with \eqref{eq:resolventdiff} yields
	\begin{align*}
		e^{-\zeta L}\mathbf{b} - Ve^{-\zeta L_{k+1}}V^\dagger\mathbf{b} &= \frac{1}{2\pi i}\int_{i\R} e^{\zeta z}(z I+L)^{-1}r_\Xi(L) r_\Xi(-z)^{-1}\mathbf{b}\,dz \\
		&- \frac{1}{2\pi i}V\int_{i\R}e^{\zeta z}(z I_{k+1}+L_{k+1})^{-1}r_\Xi(L_{k+1})r_\Xi(-z)^{-1}V^\dagger\mathbf{b} \, dz
        \\
        &= h(\zeta, L) \mathbf b - V h(\zeta, L_{k+1}) V^\dagger \mathbf b
	\end{align*}
    with $r_\Xi$ as in \eqref{eq:product} and
	\begin{align*}
		h(\zeta,\lambda) := \frac{1}{2\pi i}\int_{i\R} e^{\zeta z}(z +\lambda)^{-1}r_\Xi(\lambda)r_\Xi(-z)^{-1}\, dz.
	\end{align*}
    Crouzeix's estimate \eqref{eq:Crouzeix} and the fact that all rational Ritz values are contained in $\Sigma$ reveal
	\begin{align*}
        \norm{e^{-\zeta L}\mathbf{b} - Ve^{-\zeta L_{k+1}}V^\dagger\mathbf{b}} \leq 4C\norm{\mathbf{b}} \norm{h(\zeta,\cdot)}_{\Sigma}.
	\end{align*}
	As shown in the proof of \cite[Theorem 2]{Robol2020}, there holds
	\begin{align*}
		|h(\zeta,\lambda)| \leq 2\gamma_{k}|r_\Xi(\lambda)|,
	\end{align*}
	which concludes the proof.
\end{proof}

\begin{proof}[Proof of Lemma \ref{lemma:calpha}]
	We use \eqref{eq:calpha} to deduce
	\begin{align*}
		c_f \leq c_\alpha\int_{i\R}\frac{1}{1+t^\alpha|\zeta|^s}\frac{1}{|\zeta+\lambda_{\min}|}\,d\zeta &\le 2c_\alpha\int_1^\infty \frac{d\zeta}{\zeta + t^\alpha\zeta^{1+s}} + 2c_\alpha\int_0^1 \frac{d\zeta}{\lambda_{\min}} \\
		&= 2c_\alpha\int_1^\infty\frac{d\zeta}{\zeta^{1+s}(\zeta^{-s}+t^\alpha)} + \frac{2c_\alpha}{\lambda_{\min}}. 
	\end{align*}
	Using the substitution $\xi = \zeta^{-s} + t^\alpha$ reveals
	\[
		\int \frac{d\zeta}{\zeta^{1+s}(\zeta^{-s}+t^\alpha)} = -\frac{\ln(\zeta^{-s} + t^\alpha)}{s} \;\xrightarrow{\zeta\to\infty}\; -\frac{\ln(t^\alpha)}{s}
    \]
	and thus
	\[
		c_f\leq \frac{2c_\alpha}{s}\left(-\ln(t^\alpha) + \ln(1+t^\alpha)\right) + \frac{2c_\alpha}{\lambda_{\min}} = 2c_\alpha\left( \frac{1}{s}\ln\left(\frac{1+t^{\alpha}}{t^\alpha}\right) + \frac{1}{\lambda_{\min}}\right).
        \qedhere
    \]
\end{proof}

\section*{Acknowledgements}
The first author has been funded by the Austrian Science Fund (FWF) through grant number F 65 and W1245. The second author has been partially supported by the Austrian Science Fund (FWF)
grant P 33956-NBL.

\bibliographystyle{plainnat}
\bibliography{Bibliography}

\end{document}